\documentclass[preprint,authoryear,12pt]{elsarticle}

\usepackage{enumerate,color}
\definecolor{labelkey}{rgb}{0.6,0,1}

\usepackage{graphicx}
\graphicspath{{Figures/}}
\usepackage{bm}
 \usepackage{amsmath,amsfonts,amssymb,mathrsfs,amsthm}
 \usepackage{mathabx}

\usepackage{lineno} 

 \usepackage{booktabs}
 \usepackage{todonotes}
 \usepackage{bbm}

\newtheorem{assumption}{Assumption}
\newtheorem{corollary}{Corollary}
\newtheorem{definition}{Definition}
\newtheorem{example}{Example}
\newtheorem{lemma}{Lemma}
\newtheorem{proposition}{Proposition}

\newtheorem{theorem}{Theorem}

\numberwithin{assumption}{section}
\numberwithin{corollary}{section}
\numberwithin{equation}{section}
\numberwithin{definition}{section}
\numberwithin{example}{section}
\numberwithin{lemma}{section}
\numberwithin{proposition}{section}
\numberwithin{remark}{section}
\numberwithin{theorem}{section}

\renewcommand{\leq}{\leqslant}

\renewcommand{\le}{\leqslant}
\renewcommand{\ge}{\geqslant}

\newcommand{\clr}{\mathrm{clr}}
\DeclareMathOperator{\argmin}{argmin}
\DeclareMathOperator{\prop2}{\,\propto\,}

\journal{arXiv}
\begin{document}
\begin{frontmatter}

\title{Orthogonal decomposition of multivariate densities \\
in Bayes spaces and its connection with copulas}

\author[label1]{Christian Genest}
\address[label1]{McGill University, Montr\'eal (Qu\'ebec) Canada}

\author[label2]{Karel Hron}
\address[label2]{Palack\'y University Olomouc, Olomouc, Czech Republic}

\author[label1]{Johanna G. Ne$\check{\rm s}$lehov\'a}

\begin{abstract}
Bayes spaces were initially designed to provide a geometric framework for the modeling and analysis of distributional data. It has recently come to light that this methodology can be exploited to provide an orthogonal decomposition of bivariate probability distributions into an independent and an interaction part. In this paper, new insights into these results are provided by reformulating them using Hilbert space theory and a multivariate extension is developed using a distributional analog of the Hoeffding--Sobol identity. A connection between the resulting decomposition of a multivariate density and its copula-based representation is also provided. 
\end{abstract}


\end{frontmatter}

\section{Introduction\label{sec:1}}

In analyzing compositional data and, more generally, observations expressible as a function defined on a finite set or on an interval, the interest often focuses on the shape of the response, in recognition of the fact that the information carried by the observed function is relative rather than absolute in nature. It is then natural to cast the analysis in a space of functions in which maps $f$ and $c \times f$ for arbitrary strictly positive real number $c \in (0, \infty)$ belong to the same equivalence class.

The notion of a Bayes space, reviewed in Section~\ref{sec:2}, is tailor-made for this purpose. Originally introduced by \cite{Egozcue/etal:2006} with a linear space structure, it was further extended to a full Hilbert space by \cite{Boogaart/etal:2010}, \cite{Egozcue/etal:2013}, and \cite{Boogaart/etal:2014}, who endowed it with a notion of scalar product which is applicable to any strictly positive function whose logarithm is square integrable with respect to a reference measure $\lambda$ on a space $\Omega$. They also highlighted the fact that through a centered log ratio (clr) transformation, also recalled in Section~\ref{sec:2}, a Bayes space $\mathcal{B}^2 (\Omega)$ is isomorphic to the classical $L_0^2 (\Omega)$ space of square-integrable functions which integrate to zero on $\Omega$.

Starting with the work of \cite{Delicado:2011}, Bayes spaces have slowly emerged as a tool for the analysis of probability density functions on the real line. See, among others, \cite{Hron/etal:2016}, \cite{Talska/etal:2018,Talska/etal:2020,Talska/etal:2021} for applications to age distribution, newborn screening, income distribution, and particle size distribution, respectively. See also \cite{Seo/Beare:2019} for an analysis of wage data, as well as \cite{Menafoglio/etal:2014,Menafoglio/etal:2016,Menafoglio/etal:2018,Menafoglio/etal:2021} for numerous applications in geosciences. 

In their recent review of Bayes space methodology and alternative techniques for the analysis of probability density functions as data objects, \cite{Petersen/etal:2022}  mentioned that multivariate densities are also frequently encountered in practice and that extending the methods they reviewed to this setting is a non-trivial task. As these authors point out, the Bayes space representation provides a sound theoretical base for this task. Indeed \cite{Hron/etal:2020} recently showed that on the Bayes space $\mathcal{B}^2 (\Omega_1 \times \Omega_2)$ endowed with a product reference measure $\lambda = \lambda_1 \times \lambda_2$ which is a probability measure, any element can be decomposed into three mutually orthogonal terms, namely two geometric margins whose direct sum constitutes the independence part, and a third term representing an interaction. Their paper also includes an illustration to bivariate anthropometric data.

In this paper, the work of \cite{Hron/etal:2020} is revisited and generalized in two different ways. After a summary of basic facts about general Bayes spaces in Section~\ref{sec:2}, the bivariate approach of \cite{Hron/etal:2020} is extended in Section~\ref{sec:3} to arbitrary finite reference measure $\lambda$, and the terms of the decomposition are shown to correspond to projections in mutually orthogonal subspaces. In Section~\ref{sec:4}, a $d$-variate extension of the orthogonal decomposition of \cite{Hron/etal:2020} is then obtained from an analog of the Hoeffding--Sobol decomposition due to \cite{Kuo/etal:2010}; see also \cite{Mercadier/etal:2022}. As in the bivariate case, the decomposition is proved to be a direct sum of projections into mutually orthogonal subspaces of $\mathcal{B}^2 (\lambda)$. The $2^d - 1$ terms in the general decomposition are akin to $d$ main effects and $2^d - d -1$ multi-way interactions in analysis of variance.

In Section~\ref{sec:5}, two choices of reference measure $\lambda$ are considered. On the one hand, the use of Lebesgue measure elucidates the connection between the geometric margins of the Bayes space decomposition of a probability density function and its margins. On the other hand, the notion of Fr\'echet space of multivariate probability distributions \citep{Joe:1997} suggests the use of the product of the marginal distributions as an alternative reference measure $\lambda$. This leads to an orthogonal decomposition of the unique underlying copula density associated to an element of the Bayes space. As an illustration of the insights provided by this reduction, it is shown that the only non-trivial terms in the orthogonal decomposition of the multivariate Gaussian copula are those which correspond to 
the main effects and two-way interactions. Conclusions and perspectives for future work can be found in Section~\ref{sec:6}.

\section{Bayes spaces\label{sec:2}}

This section gives an overview of the theory of Bayes spaces developed by \cite{Boogaart/etal:2010,Boogaart/etal:2014} and summarizes the key results needed for the orthogonal decomposition of probability densities initiated by~\cite{Hron/etal:2020} in the bivariate case and generalized here.

Consider an arbitrary measurable space $(\Omega, \mathcal{A})$ and a $\sigma$-finite, positive, real-valued reference measure $\lambda$ on $(\Omega, \mathcal{A})$. Let $\mathcal{M}(\lambda)$ be the set of all $\sigma$-finite, positive, real-valued measures $\mu$ on $(\Omega,\mathcal{A})$ which are equivalent to $\lambda$, meaning that for all $A \in \mathcal{A}$,
\[
\mu(A) = 0 \quad \Longleftrightarrow \quad \lambda(A) = 0.
\]

Motivated by the fact that proportional likelihoods carry the same amount of information, \cite{Boogaart/etal:2010} set, for any two measures $\nu$, $\mu \in \mathcal{M}(\lambda)$,
\begin{equation}
\label{eq:2.1}
\nu \prop2 \mu \quad \Longleftrightarrow \quad \exists_{c \in (0, \infty)} \quad \forall_{A \in \mathcal{A}} \quad \nu(A) = c \, \mu(A).
\end{equation}
This relation being reflexive, symmetric, and transitive, it defines an equivalence relation on $\mathcal{M}(\lambda)$.

\begin{definition}
\label{def:2.1}
The Bayes space $\mathcal{B}(\lambda)$ is the quotient space $\mathcal{M}(\lambda)|{\prop2}$, i.e., the space of equivalence classes of measures in $\mathcal{M}(\lambda)$ with respect to the relation~$\prop2$. 
\end{definition}

\subsection{Operations on $\mathcal{B}(\lambda)$\label{sec:2.1}}

In their paper, \cite{Boogaart/etal:2010}  introduce two operations on $\mathcal{M}(\lambda)$, namely perturbation $\oplus$ and powering $\odot$. The latter are inspired by analogous operations on the simplex that give rise to the Aitchison geometry underlying the analysis of compositional data.

\begin{definition}
\label{def:2.2}
Given arbitrary measures $\mu$, $\nu \in \mathcal{M}(\lambda)$, let $f_\mu$, $f_\nu$ denote their Radon--Nikodym derivatives with respect to $\lambda$. Then, for any $A \in \mathcal{A}$ and $\alpha \in \mathbb{R}$,
\begin{equation}
\label{eq:2.2}
(\mu \oplus \nu) (A) = \int_A f_\mu f_\nu  d \lambda, \quad (\alpha \odot \mu) (A) = \int_A (f_\mu)^\alpha d \lambda.
\end{equation}
By convention, $\mu \ominus \nu = \mu \oplus ((-1) \odot \nu))$.
\end{definition}

\cite{Boogaart/etal:2010} show that $\mu \oplus \nu$ and $\alpha \odot \mu$ are elements of $\mathcal{M}(\lambda)$. They also prove, inter alia, that $(\mathcal{B}(\lambda), \oplus, \odot)$ is a real vector space. 

\subsection{Interpretation in terms of $\lambda$-densities\label{sec:2.2}}

Because each measure $\mu \in \mathcal{M}$ can be identified with its Radon--Nikodym derivative $f_\mu$ with respect to $\lambda$, or $\lambda$-density for short, the space $(\mathcal{B}(\lambda), \oplus, \odot)$ can be understood as the quotient space of $\lambda$-densities of measures in $\mathcal{M}(\lambda)$ with respect to the equivalence relation $\prop2$. The analog of Eq.~\eqref{eq:2.1} is then
\[
f_\mu \prop2 f_\nu \quad \Longleftrightarrow \quad \exists_{c \in (0, \infty)} \quad f_\mu = c f_\nu. 
\]
The perturbation and powering operations are then understood as operations on $\lambda$-densities, i.e., for any $\alpha \in \mathbb{R}$,
\[
f_\mu \oplus f_\nu \prop2 f_\mu f_\nu, \quad \alpha \odot f_\mu \prop2 (f_\mu)^{\alpha}.
\]
which is the analog of Eq.~\eqref{eq:2.1}. Because of this equivalence, the measure- and density-based characterizations of the Bayes space $\mathcal{B}(\lambda)$ will be used interchangeably.

\subsection{Hilbert space structure\label{sec:2.3}}

For the application of Bayes spaces considered in this paper, it is important to endow $\mathcal{B}(\lambda)$ with a scalar product that would turn it into a Hilbert space. However, two restrictions need to be made to accomplish this, namely
\begin{itemize}
\item [(i)]
$\lambda$ needs to be a finite measure;
\item [(ii)]
$\mathcal{B}(\lambda)$ needs to be restricted to the subspace $\mathcal{B}^2 (\lambda)$ of $\lambda$-densities whose logarithm is square integrable, viz.
\[
\mathcal{B}^2 (\lambda) = \left\{ f_\mu \in \mathcal{B}(\lambda) \,{\Big |} \int_\Omega \left|\ln (f_\mu) \right|^2  d \, \lambda < \infty \right\}.
\]
\end{itemize}
These assumptions are henceforth taken to hold.

By Proposition~1 in \cite{Boogaart/etal:2014}, $(\mathcal{B}^2 (\lambda),\oplus, \odot)$ is a vector subspace of $(\mathcal{B}(\lambda),\oplus, \odot)$. Having made these restrictions, one can define a scalar product on $\mathcal{B}^2 (\lambda)$ as proposed by~\cite{Talska/etal:2020}. 

\begin{definition}
\label{def:2.3}
For $f_\mu$, $f_\nu \in \mathcal{B}^2 (\lambda)$, set
\begin{equation}
\label{eq:2.3}
\langle f_\mu,f_\nu\rangle_{\mathcal{B}^2 (\lambda)} = \frac{1}{2\mathtt{\lambda}(\Omega)} \int_\Omega \int_\Omega \ln \left\{ \frac{f_\mu(\omega_1)}{f_\mu(\omega_2)} \right\} \ln \left\{ \frac{f_\nu(\omega_1)}{f_\nu(\omega_2)} \right\} d \,\lambda(\omega_1) d\, \lambda(\omega_2).
\end{equation}
\end{definition}

It then follows that $(\mathcal{B}^2 (\lambda), \langle \cdot, \cdot \rangle_{\mathcal{B}^2 (\lambda)})$ is a separable Hilbert space, as stated in Theorem~1 of \cite{Boogaart/etal:2014}, albeit with a somewhat different definition of the scalar product involving the factor $1/\{2\lambda^2(\Omega)\}$ instead of $1/\{2\lambda(\Omega)\}$. The change in definition proposed by \cite{Talska/etal:2020} extends the ideas of \cite{Egozcue/etal:2016} and mimics subcompositional dominance in compositions. The two definitions coincide when $\lambda$ is a probability measure.

The scalar product in \eqref{eq:2.3} induces a norm and a distance in the usual way. That is, for arbitrary $f_\mu, f_\nu \in \mathcal{B}^2 (\lambda)$, one sets
\[
||f_\mu||_{\mathcal{B}^2 (\lambda)} = \sqrt{\langle f_\mu, f_\mu\rangle_{\mathcal{B}^2 (\lambda)}},\quad d_{\mathcal{B}^2 (\lambda)}(f_\mu,f_\nu)=||f_\mu\ominus f_\nu||_{\mathcal{B}^2 (\lambda)},
\]
where $f_\mu\ominus f_\nu=f_\mu\oplus[(-1)\odot f_\nu]$ is the perturbation-subtraction of densities. 

\subsection{clr transformation\label{sec:2.4}}

As explained by \cite{Boogaart/etal:2014}, the Bayes space $\mathcal{B}^2 (\lambda)$ is in one-to-one correspondence with the subspace  
\begin{equation}\label{eq:2.4}
L_0^2 (\lambda) = \left\{ f : \Omega \to \mathbb{R} \, {\Big |} \int_\Omega f^2 d \lambda < \infty, \int_\Omega f d\lambda = 0 \right\}
\end{equation}
of the classical Hilbert space $L^2 (\lambda)$ of square integrable functions with respect to $\lambda$. This correspondence is provided by the mapping ${\rm clr}: \mathcal{B}^2 (\lambda) \to L_0^2 (\lambda)$, defined for every $f_\mu \in \mathcal{B}^2 (\lambda)$, by
\[
\clr (f_\mu) = \ln (f_\mu) - \frac{1}{\lambda({\Omega})} \int_\Omega \ln (f_\mu) d \lambda
\]
and known as the centred log-ratio transformation, or clr transformation for short. 

It is easy to see from Theorem~1 in \cite{Boogaart/etal:2014} that with the definition of the scalar product in \eqref{eq:2.3}, the clr transformation is an isomorphism. In particular, the map clr is linear, i.e., for any $f_\mu$, $f_\nu$ in $\mathcal{B}^2 (\lambda)$ and $\alpha \in \mathbb{R}$, one has
\[
\clr (f_\mu \oplus f_\nu) = \clr (f_\mu) + \clr (f_\nu), \quad \clr (\alpha \odot f_\mu) =\alpha \cdot \clr (f_\mu).
\]
In the context of Bayes spaces, the space $L_0^2 (\lambda) $ is also referred to as the clr space.

\section{Orthogonal decomposition of bivariate densities\label{sec:3}}

Up to this point, the reference measure $\lambda$ needed only be real-valued, positive, and finite. Focusing on the case where $(\Omega, \mathcal{A})$ is a product space and $\lambda$ is a product of probability measures, \cite{Hron/etal:2020} used Bayes space methodology to develop an orthogonal decomposition of bivariate probability densities in $\mathcal{B}^2 (\lambda)$. 

In this section, the findings of \cite{Hron/etal:2020} are reviewed and recast in a more general framework which accommodates arbitrary positive finite reference measures $\lambda$ rather than probability measures. Moreover, it is shown that the elements of the decomposition are in fact orthogonal projections on certain subspaces of $\mathcal{B}^2 (\lambda)$. 

\subsection{Notation\label{sec:3.1}}

Let $(\Omega_1, \mathcal{A}_1, \lambda_1)$ and $(\Omega_2, \mathcal{A}_2, \lambda_2)$ be two measurable spaces with finite, positive, real-valued measures $\lambda_1$ and $\lambda_2$, respectively. Set
\[
\Omega = \Omega_1 \times \Omega_2, \quad \mathcal{A} = \mathcal{A}_1 \otimes \mathcal{A}_2, \quad \lambda = \lambda_1 \otimes \lambda_2.
\]
Let $\mathcal{B}(\lambda)$ be the Bayes space on $(\Omega,\mathcal{A})$ and $\mathcal{B}^2 (\lambda)$ be the subspace of $\lambda$-densities whose logarithm is square-integrable, as in Section~\ref{sec:2.3}. For each $i \in \{1, 2 \}$, let
\begin{itemize}
\item [(i)]
$\mathcal{B}(\lambda_i)$ be the Bayes space on $(\Omega_i, \mathcal{A}_i)$ with reference measure $\lambda_i$;
\item [(ii)]
$\mathcal{B}^2(\lambda_i)$ be the subspace of $\mathcal{B}(\lambda_i)$ consisting of $\lambda_i$-densities whose logarithm is square-integrable. 
\end{itemize}

The spaces $\mathcal{B}^2(\lambda_1)$ and $\mathcal{B}^2(\lambda_2)$ can be embedded in $\mathcal{B}^2 (\lambda)$ as follows: any $\mu_1 \in  \mathcal{B}^2(\lambda_1)$ is mapped to the product measure $\mu_1 \otimes \lambda_2$ and similarly, any $\mu_2 \in \mathcal{B}^2(\lambda_2)$ is mapped to $\lambda_1 \otimes \mu_2$. These embeddings are respectively denoted $\mathcal{B}^2_{ \{ 1 \} } (\lambda)$ and $\mathcal{B}^2_{ \{ 2 \} } (\lambda)$. 

In what follows, $L_0^2 (\lambda)$, $L_0^2(\lambda_1)$, and $L_0^2(\lambda_2)$ refer to the subspaces of the respective $L_2$ spaces on $\Omega$, $\Omega_1$, and $\Omega_2$ as specified in Eq.~\eqref{eq:2.4}.

\subsection{Properties of the subspaces $\mathcal{B}^2_{ \{ 1 \} } (\lambda)$ and $\mathcal{B}^2_{ \{ 2 \} } (\lambda)$\label{sec:3.2}}

The key to the orthogonal decomposition of bivariate densities proposed by \cite{Hron/etal:2020} is the fact that the subspaces $\mathcal{B}^2_{ \{ 1 \} } (\lambda)$ and $\mathcal{B}^2_{ \{ 2 \} } (\lambda)$ of $\mathcal{B}^2 (\lambda)$ are complete and orthogonal. To prove this result, which is formally stated below, the following technical result will be needed.

\begin{lemma}
\label{lem:3.1}
Consider arbitrary $f, g \in \mathcal{B}^2_{\rm ind}(\lambda)$ and suppose that $f \prop2 f_1 f_2$ and $g \prop2 g_1 g_2$, where for $j \in \{1,2\}$, $f_j, g_j$ are $\lambda_j$-densities. Then 
\begin{enumerate}
\item[(i)] $\langle f ,g  \rangle_{\mathcal{B}^2 (\lambda)} = \lambda_2(\Omega_2) \langle f_1,g_1  \rangle_{\mathcal{B}^2(\lambda_1)} + \lambda_1(\Omega_1) \langle f_2,g_2  \rangle_{\mathcal{B}^2(\lambda_2)}$.
\item[(ii)] $\| f\|^2_{\mathcal{B}^2 (\lambda)} = \lambda_2(\Omega_2) \| f_1\|^2_{\mathcal{B}^2(\lambda_1)} +  \lambda_1(\Omega_1) \| f_2\|^2_{\mathcal{B}^2(\lambda_2)}$.
\end{enumerate}
\end{lemma}
 
\begin{proof} As statement (ii) is a direct consequence of statement (i), it suffices to prove the latter. To this end, first observe that the definition of the clr transformation immediately implies that
\[
\clr (f) = \clr (f_1) + \clr (f_2),
\]
where for $i \in \{1,2\}$, $\clr (f_i)$ refers to the transformation of $f_i$ by the clr mapping from $\mathcal{B}^2(\lambda_j)$ to $L^2_0(\lambda_j)$. An analogous result holds for $g$, viz. $\clr (g) = \clr (g_1) + \clr (g_2)$.  Second, note that because for any $i \in \{1,2\}$, $\clr (f_i)$ and $\clr (g_i)$ are elements of $L_0^2(\lambda_j)$ and in particular integrate to $0$, one has, for any distinct $i, j \in \{1,2\}$,
\[
\langle \clr (f_i), \clr (g_j) \rangle_{L^2 (\lambda)} = \int_{\Omega_i} \clr (f_i) d \lambda_i \int_{\Omega_j} \clr (g_j) d \lambda_j = 0. 
\]
These observations and the fact that the clr transformation is an isomorphism imply
\begin{align*}
\langle f ,g  \rangle_{\mathcal{B}^2 (\lambda)} & = \langle \clr (f),\clr (g)  \rangle_{L^2 (\lambda)} = \langle \clr (f_1) + \clr (f_2),\clr (g_1)+\clr (g_2)  \rangle_{L^2 (\lambda)} \\
&= \langle \clr (f_1) ,\clr (g_1)  \rangle_{L^2 (\lambda)} + \langle \clr (f_2),\clr (g_2)  \rangle_{L^2 (\lambda)} \\
& = \lambda_2(\Omega_2) \langle \clr (f_1) ,\clr (g_1)  \rangle_{L^2(\lambda_1)} + \lambda_1(\Omega_1) \langle \clr (f_2),\clr (g_2)  \rangle_{L^2(\lambda_2)}\\
 & = \lambda_2(\Omega_2) \langle f_1 ,g_1  \rangle_{\mathcal{B}^2(\lambda_1)} + \lambda_1(\Omega_1) \langle f_2,g_2  \rangle_{\mathcal{B}^2(\lambda_2)}
\end{align*} 
as was to be shown.
\end{proof}

\begin{lemma}
\label{lem:3.2} 
The spaces $\mathcal{B}^2_{ \{ 1 \} } (\lambda)$ and $\mathcal{B}^2_{ \{ 2 \} } (\lambda)$ are complete and orthogonal.
\end{lemma}

\begin{proof}
For any $f \in \mathcal{B}^2_{ \{ 1 \} } (\lambda)$, one has $f \prop2 f_1 \cdot 1$, where $f_1$ is a $\lambda_1$-density and $1$ is the neutral element in $\mathcal{B}^2 (\lambda_2)$. It follows from Lemma~\ref{lem:3.1} that
\[
\| f\|^2_{\mathcal{B}^2 (\lambda)} = \lambda_2(\Omega_2) \| f_1\|^2_{\mathcal{B}^2(\lambda_1)}.
\]
Therefore, the completeness of $\mathcal{B}^2_{ \{ 1 \} } (\lambda)$ follows directly from that of $\mathcal{B}^2 (\lambda_1)$ shown by \cite{Boogaart/etal:2014}. The argument for $\mathcal{B}^2 (\lambda_2)$ is analogous.

The fact that the subspaces $\mathcal{B}^2_{ \{ 1 \} } (\lambda)$ and $\mathcal{B}^2_{ \{ 2 \} } (\lambda)$ are orthogonal also follows from Lemma~\ref{lem:3.1}. Indeed, for arbitrary $f \in \mathcal{B}^2_{ \{ 1 \} } (\lambda)$ and $g \in \mathcal{B}^2_{ \{ 2 \} } (\lambda)$, one has $f \prop2 f_1 \cdot 1$ and $g \prop2 1 \cdot g_2$ with $f_1 \in \mathcal{B}^2 (\lambda_1)$ and $g_2 \in \mathcal{B}^2 (\lambda_2)$. Thus
\[
\langle f, g \rangle_{\mathcal{B}^2 (\lambda)} = \lambda_2 (\Omega_2) \langle f_1,1  \rangle_{\mathcal{B}^2(\lambda_1)} + \lambda_1 (\Omega_1) \langle 1,g_2  \rangle_{\mathcal{B}^2(\lambda_2)} = 0.
\]
This concludes the argument.
\end{proof}

The first component in the orthogonal decomposition of bivariate densities proposed by~\cite{Hron/etal:2020} is the so-called independence space, defined by
\begin{equation}
\label{eq:3.1}
\mathcal{B}_{\rm ind}^2 (\lambda) = \{ \mu \in \mathcal{B}^2 (\lambda) \;| \;  \exists_{\mu_1 \in \mathcal{B}^2(\lambda_1)} \; \exists_{\mu_2 \in \mathcal{B}^2(\lambda_2)} \; \mu = \mu_1 \otimes \mu_2\}.
\end{equation}
From well-known results about product spaces, $\mathcal{B}_{\rm ind}^2 (\lambda)$ can be viewed as the space of $\lambda$-densities, which are products of a $\lambda_1$-density and a $\lambda_2$-density. As stated below, this space is complete and related in a simple way to the spaces $\mathcal{B}^2_{ \{ 1 \} } (\lambda)$ and $\mathcal{B}^2_{ \{ 2 \} } (\lambda)$.

\begin{lemma}
\label{lem:3.3} 
The space $\mathcal{B}_{\rm ind}^2 (\lambda)$ of $\mathcal{B}^2 (\lambda)$ is complete and $\mathcal{B}^2_{ \{ 1 \} } (\lambda) \oplus \mathcal{B}^2_{ \{ 2 \} } (\lambda) = \mathcal{B}_{\rm ind}^2 (\lambda)$.  
\end{lemma}

\begin{proof}
The completeness of $\mathcal{B}_{\rm ind}^2 (\lambda)$ in $\mathcal{B}^2 (\lambda)$ is a direct consequence of Lemma~\ref{lem:3.1} (ii) and the completeness of the spaces $\mathcal{B}^2(\lambda_1)$ and $\mathcal{B}^2(\lambda_2)$. The fact that $\mathcal{B}_{\rm ind}^2 (\lambda)$ is the direct sum of $\mathcal{B}^2_{ \{ 1 \} } (\lambda)$ and $\mathcal{B}^2_{ \{ 2 \} } (\lambda)$ follows from their orthogonality shown in Lemma~\ref{lem:3.2} and the property that densities of product measures are products of densities, i.e., that any $f \in \mathcal{B}_{\rm ind}^2 (\lambda)$ satisfies $f \prop2 f_1 f_2$, where $f_1$ is a $\lambda_1$-density and $f_2$ is a $\lambda_2$-density. Thus $f = (f_1 \cdot 1) \oplus (1 \cdot f_2)$ is a sum of two elements from $\mathcal{B}^2_{ \{ 1 \} } (\lambda)$ and $\mathcal{B}^2_{ \{ 2 \} } (\lambda)$, respectively.
\end{proof}

\subsection{Geometric margins\label{sec:3.3}}
  
Mimicking \cite{Hron/etal:2020}, one can define the geometric margins of an element $f_\mu \in \mathcal{B}^2 (\lambda)$ to be the $\lambda$-densities $f_{\mu,1} \in \mathcal{B}^2_{\{ 1 \}} (\lambda)$ and $f_{\mu,2} \in \mathcal{B}^2_{\{ 2 \}} (\lambda)$ given, for any $(\omega_1,\omega_2) \in \Omega$, by
\[
f_{\mu,1} (\omega_1,\omega_2) \equiv f_{\mu,1} (\omega_1) \prop2 \exp \left[ \frac{1}{\lambda_2 (\Omega_2)} \int_{\Omega_2} \ln \{ f_\mu(\omega_1,\omega^*) \} d \lambda_2(\omega^*) \right]
\]
and
\[
f_{\mu,2}(\omega_1,\omega_2) \equiv f_{\mu,2} (\omega_2) \prop2 \exp \left[ \frac{1}{\lambda_1(\Omega_1)} \int_{\Omega_1} \ln \{ f_\mu(\omega^*,\omega_2) \} d \lambda_1(\omega^*) \right].
\]

These definitions are more general than those of \cite{Hron/etal:2020}, which were restricted to the case where $\lambda_1$ and $\lambda_2$ are probability measures. It will be seen below that when $\lambda_1$ and $\lambda_2$ are merely finite, the factors $1/\lambda_2(\Omega_2)$ and $1/\lambda_1(\Omega_1)$ must be included so that $f_{\mu,1}$ and $f_{\mu,2}$ can be interpreted as orthogonal projections of $f_\mu$ onto $\mathcal{B}^2_{ \{ 1 \} } (\lambda)$ and $\mathcal{B}^2_{ \{ 2 \} } (\lambda)$, respectively. 

Before the issue of orthogonal projections can be addressed, it is useful to compute the clr transformation of the geometric margins and to see how it relates to the clr transformation of $f_\mu$. 
For each $i \in \{1, 2 \}$, the clr transformation of the geometric margin $f_{\mu,i}$, denoted $\clr (f_{\mu,i})$, is a function of a single variable, but in contrast to \cite{Hron/etal:2020}, the factor $1/\lambda_i (\Omega_i)$ is still present.

\begin{lemma}
\label{lem:3.4}
Given any $f_\mu \in \mathcal{B}^2 (\lambda)$ with geometric margins $f_{\mu,1}$ and $f_{\mu,2}$, one has
\[
\clr (f_{\mu,1}) = \frac{1}{\lambda_2 (\Omega_2)} \int_{\Omega_2} \clr (f_\mu) d \lambda_2, \quad \clr (f_{\mu,2}) =  \frac{1}{\lambda_1 (\Omega_1)} \int_{\Omega_1} \clr (f_\mu) d \lambda_1.
\]
\end{lemma} 
\begin{proof}
To establish the first identity, write
\begin{multline*}
\clr (f_{\mu,1})(\omega_1,\omega_2) = \frac{1}{\lambda_2(\Omega_2)} \int_{\Omega_2} \ln \{ f_\mu(\omega_1,\omega^*) \} d \lambda_2(\omega^*)  \\
 - \frac{1}{\lambda_2(\Omega_2)\lambda(\Omega)} \int_{\Omega} \int_{\Omega_2}\ln \{f_\mu (\omega_1^*,\omega^*)\} d \lambda_2(\omega^*) d \lambda_1 (\omega_1^*) d \lambda_2(\omega_2^*) .
 \end{multline*}
Then observe that by Fubini's theorem, this expression can be rewritten as
\[
 \frac{1}{\lambda_2(\Omega_2)}\int_{\Omega_2}  \Bigl[ \ln \{ f_\mu(\omega_1,\omega^*) \} - \frac{1}{\lambda(\Omega)} \int_{\Omega} \ln \{f_\mu (\omega_1^*,\omega_2^*)\}   d \lambda_1(\omega_1^*) d \lambda_2(\omega_2^*) \Bigr] d \lambda_2(\omega^*). 
\]
The term in the square brackets is precisely $\clr (f_\mu)$, as was to be shown. The second identity can be proved analogously.
\end{proof}

\subsection{Geometric margins as projections\label{sec:3.4}}

The following result provides a key insight into what geometric margins represent.

\begin{proposition}
\label{prop:3.1}
For arbitrary $f_\mu \in \mathcal{B}^2 (\lambda)$ and $i \in \{1,2\}$, $f_{\mu,i}$ is the unique orthogonal projection of $f_{\mu}$ on $\mathcal{B}^2_{ \{ i \} } (\lambda)$. In particular,
\[
f_{\mu,i} = \argmin_{f \in \mathcal{B}_{\{ i \}}^2 (\lambda)} \| f_{\mu} \ominus f \|_{\mathcal{B}^2 (\lambda)}.
\]
\end{proposition}

\begin{proof}
First note that because $\mathcal{B}^2_{ \{ 1 \} } (\lambda)$ is a complete subspace of $\mathcal{B}^2 (\lambda)$ by Lemma~\ref{lem:3.2}, there exists a unique orthogonal projection of $f_\mu$ onto $\mathcal{B}^2_{ \{ 1 \} } (\lambda)$. To verify that $f_{\mu,1}$ is this projection, it suffices to take an arbitrary $g \in \mathcal{B}^2_{ \{ 1 \} } (\lambda)$ and prove that
\[
\langle f_\mu \ominus f_{\mu,1}, g \rangle_{\mathcal{B}^2 (\lambda)} = 0. 
\]
Given that $g \in \mathcal{B}^2_{ \{ 1 \} } (\lambda)$, one has $g \prop2 g_1 \cdot 1$ for some $\lambda_1$-density $g_1$, where $1$ denotes the neutral element in $\mathcal{B}^2(\lambda_2)$. From the proof of Lemma~\ref{lem:3.1} and the fact that the clr transformation is an isomorphism, one has
\begin{multline*}
\langle f_\mu \ominus f_{\mu,1}, g \rangle_{\mathcal{B}^2 (\lambda)}=  \langle \clr (f_{\mu}) - \clr (f_{\mu,1}), \clr (g_1) \rangle_{L_2} \\= \int_{\Omega_1} \{\clr (g_1)\}(\omega_1) \left\{ \int_{\Omega_2} \clr (f_{\mu})(\omega_1,\omega_2) d \, \lambda_2(\omega_2)  - \lambda_2(\Omega_2) \clr (f_{\mu,1})(\omega_1) \right\} d \lambda_1(\omega_1).
\end{multline*}
By Lemma~\ref{lem:3.4}, the term in curly brackets is $0$, which completes the argument in the case $i = 1$. The case $i = 2$ is similar. 
\end{proof}

An analog of Proposition~\ref{prop:3.1} also exists in the clr space. To formulate this result, one must first embed $L_0^2 (\lambda_1)$ and $L_0^2 (\lambda_2)$ in $L_0^2 (\lambda)$ by setting, for $i \in \{1, 2 \}$,
\[
L_{0,i}^2 (\lambda) = \{ f \in L_0^2 \; | \; \exists_{f_i \in L_0^2(\lambda_i)} \;\; \forall_{(\omega_1,\omega_2) \in \Omega} \;  f(\omega_1,\omega_2) = f_i(\omega_i)\}.
\]
Because the clr transformation is an isomorphism, it follows from Lemma~\ref{lem:3.2} that $L_{0,i}^2 (\lambda)$ are complete subspaces of $L_0^2 (\lambda)$. 

\begin{corollary}
\label{cor:3.1}
For arbitrary $f_\mu \in \mathcal{B}^2 (\lambda)$ and $i \in \{1,2\}$, $\clr (f_{\mu,i})$ is the unique orthogonal projection of $\clr (f_{\mu})$ on $L_{0,i}^2 (\lambda)$.
\end{corollary}

\subsection{Decomposition of \cite{Hron/etal:2020}\label{sec:3.5}}

Proposition~\ref{prop:3.1} sheds new light on the decomposition of $f_\mu$ into the so-called independence and interaction parts defined by \cite{Hron/etal:2020} as 
\begin{equation}
\label{eq:3.2}
f_{\mu,{\rm ind}} = f_{\mu,1} \oplus f_{\mu,2} = f_{\mu,1} f_{\mu,2}, \quad f_{\mu,{\rm int}} = f_\mu \ominus f_{\mu,\rm{ind}}= \frac{f_\mu}{f_{\mu,1} f_{\mu,2}}.
\end{equation}
Indeed, from Proposition~\ref{prop:3.1} and Lemma~\ref{lem:3.2}, $f_{\mu,1}$ and $f_{\mu,2}$ are orthogonal in the Bayes space, i.e.,
\[
\langle f_{\mu,1} ,f_{\mu, 2} \rangle_{\mathcal{B}^2 (\lambda)} = 0, \quad f_{\mu,{\rm ind}}  \in \mathcal{B}_{\rm ind}^2 (\lambda).  
\]
However, more is true: $f_{\mu,{\rm ind}}$ is in fact the unique orthogonal projection of $f_\mu$ onto $\mathcal{B}_{\rm ind}^2 (\lambda)$. This observation is formally stated and proved below.

\begin{proposition}\label{prop:3.2}
Let $f_\mu$ be an arbitrary element of $\mathcal{B}^2 (\lambda)$. Then the independence part $f_{\mu, {\rm ind}}$ is the unique orthogonal projection of $f_\mu$ on $\mathcal{B}_{\rm ind}^2 (\lambda)$.
\end{proposition} 

\begin{proof}
Because $\mathcal{B}_{\rm ind}^2 (\lambda)$ is a complete subspace of $\mathcal{B}^2 (\lambda)$ by Lemma~\ref{lem:3.2}, it suffices to show that for any $g \in \mathcal{B}_{\rm ind}^2 (\lambda)$, 
\[
\langle f_\mu \ominus f_{\mu, {\rm ind}}, g \rangle_{\mathcal{B}^2 (\lambda)} = 0. 
\]
To this end, write $g \prop2 g_1 g_2$, where $g_1$ is a $\lambda_1$-density and $g_2$ is a $\lambda_2$-density. Let $g_1^* \prop2 g_1 \cdot 1 $ and $g_2^* \prop2 1 \cdot g_2$ so that $g = g_1^* \oplus g_2^*$.  Using well-known properties of scalar products, the claim follows if one can show that
\[
\langle f_\mu , g \rangle_{\mathcal{B}^2 (\lambda)}  =  \langle  f_{\mu, {\rm ind}}, g \rangle_{\mathcal{B}^2 (\lambda)}. 
\]
On the one hand, Lemma~\ref{lem:3.1} implies that 
\[
\langle  f_{\mu, {\rm ind}}, g \rangle_{\mathcal{B}^2 (\lambda)} = \lambda_2(\Omega_2) \langle  f_{\mu, 1}, g_1 \rangle_{\mathcal{B}^2 (\lambda_1)} + \lambda_1(\Omega_1) \langle  f_{\mu, 2}, g_2 \rangle_{\mathcal{B}^2 (\lambda_2)}.
\]
On the other hand, $ \langle f_\mu , g \rangle_{\mathcal{B}^2 (\lambda)} =  \langle f_\mu , g_1^* \rangle_{\mathcal{B}^2 (\lambda)} +  \langle f_\mu , g_2^* \rangle_{\mathcal{B}^2 (\lambda)}$. However, Proposition~\ref{prop:3.1} implies that for each $i \in \{1, 2 \}$, one has
\[
\langle f_\mu , g_i^* \rangle_{\mathcal{B}^2 (\lambda)} = \langle f_\mu \ominus f_{\mu,i} , g_i^* \rangle_{\mathcal{B}^2 (\lambda)} + \langle  f_{\mu,i} , g_i^* \rangle_{\mathcal{B}^2 (\lambda)} = \langle  f_{\mu,i} , g_i^* \rangle_{\mathcal{B}^2 (\lambda)}.
\]
From Eq.~\eqref{eq:2.3}, it is immediate that $\langle  f_{\mu,i} , g_i^* \rangle_{\mathcal{B}^2 (\lambda)} = \lambda_{3-i} (\Omega_{3-i}) \langle  f_{\mu,i} , g_i \rangle_{\mathcal{B}^2 (\lambda_i)}$ and hence the argument is complete.
\end{proof}

Proposition~\ref{prop:3.2} in turn implies that by construction, the interaction part $f_{\mu,{\rm int}}$ introduced in Eq.~\eqref{eq:3.2} lies in the orthogonal complement of $\mathcal{B}_{\rm ind}^2 (\lambda)$, viz.
\begin{equation}
\label{eq:3.3}
\mathcal{B}_{\rm int}^2 (\lambda) = \{ f_\mu \in \mathcal{B}^2 (\lambda) \; | \; \forall_{f_{\nu} \in \mathcal{B}_{\rm ind}^2 (\lambda) } \; \; \langle f_\mu , f_\nu \rangle_{\mathcal{B}^2 (\lambda)} = 0 \},
\end{equation}
which can be termed the interaction space. 

\begin{corollary}
\label{cor:3.2}
 Let $f_\mu$ be an arbitrary element of $\mathcal{B}^2 (\lambda)$. Then the interaction part $f_{\mu, {\rm int}}$ is the unique orthogonal projection of $f_\mu$ on $\mathcal{B}_{\rm int}^2 (\lambda)$.
\end{corollary}

Put together, the above developments make it possible to recover the following orthogonal decomposition of bivariate densities discovered by \cite{Hron/etal:2020}:
\begin{equation}
\label{eq:3.4}
f_\mu = f_{\mu,1} \oplus f_{\mu,2} \oplus  f_{\mu,{\rm int}} = f_{\mu,{\rm ind}} \oplus f_{\mu,{\rm int}}.
\end{equation}

\subsection{Consequences of the orthogonal decomposition \eqref{eq:3.4}\label{sec:3.6}}

\cite{Hron/etal:2020} study the consequences of this decomposition in detail and derive a number of results about it from first principles. As described below, their main findings can be recovered as simple consequences of Propositions~\ref{prop:3.1} and \ref{prop:3.2}.

\begin{proposition}
\label{prop:3.3}
For any $f_\mu \in \mathcal{B}^2 (\lambda)$, the following statements hold true.
\begin{enumerate}
\item[(i)] 
\emph{Pythagoras' Theorem:} One always has
\[
\| f_\mu\|^2_{\mathcal{B}^2 (\lambda)} = \| f_{\mu,\rm{ind}}\|^2_{\mathcal{B}^2 (\lambda)} + \| f_{\mu,\rm{int}}\|^2_{\mathcal{B}^2 (\lambda)} = \| f_{\mu,1}\|^2_{\mathcal{B}^2 (\lambda)} + \| f_{\mu,2}\|^2_{\mathcal{B}^2 (\lambda)}+ \| f_{\mu,\rm{int}}\|^2_{\mathcal{B}^2 (\lambda)}.
\]
\item[(ii)] 
\emph{Margin-free property of the interaction part:} 
The geometric margins $f_{\mu, {\rm int}, 1}$ and $f_{\mu, {\rm int}, 2}$ of  $f_{\mu, {\rm int}}$ satisfy 
\[
f_{\mu, {\rm int}, 1} \prop2 1, \quad f_{\mu, {\rm int}, 2} \prop2 1.
\]
\item[(iii)] 
\emph{Independence:}
If $f_\mu = f_{\mu_1} f_{\mu_2} \in \mathcal{B}^2_{\rm ind}(\lambda)$, then $f_{\mu,1} \prop2 f_{\mu_1}$, $f_{\mu,2} \prop2 f_{\mu_2}$, and 
\[
f_\mu \in \mathcal{B}^2_{\rm ind}(\lambda) \quad \Longleftrightarrow \quad f_\mu \prop2 f_{\mu,1} \oplus f_{\mu,2} \quad \Longleftrightarrow \quad f_{\mu, {\rm int}} \prop2 1.
\] 
\item[(iv)] 
\emph{Yule perturbation:} 
For $g \prop2 g_1 g_2 \in  \mathcal{B}^2_{\rm ind}(\lambda)$, set $h = f_\mu \oplus g$. Then the independence and interaction parts of $h$ satisfy 
\[
h_{\rm ind} = f_{\mu, {\rm ind}} \oplus g, \quad h_{\rm int} = f_{\mu, {\rm int}}.
\]
In particular, the geometric margins of $h$ are $h_1\prop2 f_{\mu,1} g_1$  and $h_2 \prop2 f_{\mu,2} g_2$.
\end{enumerate}
\end{proposition}

\begin{proof}
The only claim that is not immediate from Propositions~\ref{prop:3.1} and \ref{prop:3.2} is the statement in (iii) that if $f_\mu = f_{\mu_1} f_{\mu_2} \in \mathcal{B}^2_{\rm ind}(\lambda)$, then $f_{\mu,1} \prop2 f_{\mu_1}$ and $f_{\mu,2} \prop2 f_{\mu_2}$.  It suffices to verify the first of these claims. Given that $f_\mu = f_{\mu_1} f_{\mu_2}$, one has
\[
f_{\mu, 1} = \exp \left[ \frac{1}{\lambda_2(\Omega_2)} \int_{\Omega_2} (\ln f_{\mu_1} + \ln f_{\mu_2}) d\lambda_2 \right] \prop2 \exp \left\{ \frac{1}{\lambda_2(\Omega_2)} \int_{\Omega_2} \ln (f_{\mu_1}) d\lambda_2 \right\} = f_{\mu_1},
\]
as for any $(\omega_1,\omega_2) \in \Omega$, $f_{\mu_1}$ depends on $\omega_1$ only while $f_{\mu_2}$ depends on $\omega_2$ only.
\end{proof}

The orthogonal decomposition in Eq.~\eqref{eq:3.4} has an obvious analog in the clr space, also shown by \cite{Hron/etal:2020}, viz.
\[
\clr (f_{\mu,{\rm ind}}) = \clr (f_{\mu,1}) + \clr (f_{\mu,2}), \quad \clr (f_{\mu,{\rm int}} ) = \clr (f_\mu) - \clr (f_{\mu,1}) - \clr (f_{\mu,2}).
\]
Pythagoras' theorem in Proposition~\ref{prop:3.3} (i) then becomes
\begin{align*}
\| \clr (f_\mu)\|^2_{L^2 (\lambda)} &= \|\clr (f_{\mu,\rm{ind}})\|^2_{L^2 (\lambda)} + \|\clr (f_{\mu,\rm{int}})\|^2_{L^2 (\lambda)} \\
&= \| \clr (f_{\mu,1})\|^2_{L^2 (\lambda)} + \|\clr (f_{\mu,2})\|^2_{L^2 (\lambda)}+ \|\clr (f_{\mu,\rm{int}})\|^2_{L^2 (\lambda)}.
\end{align*}
Statement (ii) in Proposition~\ref{prop:3.3} implies, among others, that if one were to decompose $f_{\mu, \rm{int}}$ into an independence and interaction part, the independence part would be proportional to $1$ while the interaction part would again be $f_{\mu, \rm{int}}$. This is why \cite{Hron/etal:2020} regard $f_{\mu, \rm{int}}$ as a margin-free component of $f_\mu$.

Finally, for any element $f_\mu \in \mathcal{B}^2 (\lambda)$ which corresponds to a finite measure $\mu$, i.e.,  
\begin{equation}\label{eq:3.5}
\int_{\Omega} f_\mu d \lambda < \infty.
\end{equation}
the (standard) arithmetic margins of $f_\mu$ are defined, for all $\omega_1 \in \Omega_1$ and $\omega_2 \in \Omega_2$, by
\[
f_{\mu [1]} (\omega_1) = \int_{\Omega_2} f_\mu (\omega_1, \omega_2) d \lambda (\omega_1, \omega_2) , \quad f_{\mu [2]} (\omega_2) = \int_{\Omega_1} f_\mu (\omega_1, \omega_2) d \lambda (\omega_1, \omega_2).
\]
If in addition $f_\mu \in \mathcal{B}^2_{\rm ind}(\lambda)$, i.e., $f_\mu = f_{\mu_1} f_{\mu_2}$, these arithmetic margins are $f_{\mu [1]} = f_{\mu_1}$ and  $f_{\mu [2]} = f_{\mu_2}$. Proposition~\ref{prop:3.3}~(iii) then has the following immediate implication. 

\begin{corollary}
\label{cor:3.3}
If $f_\mu \in \mathcal{B}_{\rm ind}^2 (\lambda)$ is such that \eqref{eq:3.5} holds, then its arithmetic and geometric margins coincide, i.e., $f_{\mu, 1} \prop2 f_{\mu [1]}$ and $f_{\mu, 2} \prop2 f_{\mu [2]}$.
\end{corollary}

\subsection{Illustration\label{sec:3.7}}

Before proceeding with a multivariate extension of decomposition \eqref{eq:3.4}, it may be useful to illustrate the concepts introduced here in a simple case where algebraically closed expressions are available. For additional numerical examples, see, e.g., \cite{Hron/etal:2020}.

\begin{example}
\label{ex:3.1}
Consider the following three-parameter bivariate beta distribution density studied by \cite{Jones:2002} corresponding to a special case of the work of~\cite{Libby/Novick:1982}. For arbitrary $\alpha_0, \alpha_1, \alpha_2 \in (0, \infty)$ and all $x_1, x_2 \in (0, 1)$, let 
\[
f_\mu (x_1, x_2) = \frac{1}{B(\alpha_0, \alpha_1, \alpha_2)} \, \frac{x_1^{\alpha_1-1}(1-x_1)^{\alpha_0 + \alpha_2-1}x_2^{\alpha_2-1}(1-x_2)^{\alpha_0+\alpha_1-1}}{(1-x_1x_2)^{\alpha_0 + \alpha_1 + \alpha_2}} \, ,
\]
where $B(\alpha_0, \alpha_1, \alpha_2)  = \Gamma (\alpha_0) \Gamma (\alpha_1) \Gamma (\alpha_2) / \Gamma (\alpha_0 + \alpha_1 + \alpha_2)$ is the generalized beta function. The univariate margins are beta distributions with parameters $(\alpha_1,\alpha_0)$ and $(\alpha_2, \alpha_0)$, respectively; the limiting case $\alpha_0 + \alpha_1 + \alpha_2 \to 0$ corresponds to independence \citep{Jones:2002} .

Using Lebesgue measure on $[0,1]^2$ as the reference measure $\lambda$, one finds through direct calculations that, for all $x_1, x_2 \in (0, 1)$,
\begin{align*}
\clr (f_\mu) (x_1,x_2) &= \ln \{f_\mu (x_1, x_2) \} -  \int_\Omega \ln \{ f_\mu (x_1, x_2)\} dx_1 dx_2 \\
& = (\alpha_1 - 1) \ln (x_1) + (\alpha_0 + \alpha_2 - 1) \ln (1 - x_1) \\
& \qquad + (\alpha_2 - 1) \ln (x_2) + (\alpha_0 + \alpha_1-1) \ln (1 - x_2)\\
& \qquad + (\alpha_0 + \alpha_1 + \alpha_2) \{ \pi^2/6 - \ln (1 - x_1x_2)\} - 4.
\end{align*}
The corresponding geometric marginals are then given, for all $x_1, x_2 \in (0, 1)$ and $j \in \{ 1, 2 \}$, by
\begin{multline*}
\clr (f_{\mu, j}) (x_j) = (\alpha_j-1) \ln (x_j) + (\alpha_0 + \alpha_j - 1) \ln (1 - x_j)\\
+ (\alpha_0 + \alpha_1 + \alpha_2) \left\{ \frac{\pi^2}{6} + \frac{(1 - x_1) \ln (1 - x_1)}{x_1} \right\} - 2.
\end{multline*}
These margins can then be used directly to compute the independent and the interactions terms, which are respectively given, for all $x_1, x_2 \in (0, 1)$, by
\begin{align*}
\clr (f_{\rm ind})(x_1, x_2) & = (\alpha_1 - 1) \ln (x_1) + (\alpha_2 - 1) \ln (x_2) \\
& \qquad + (\alpha_0 + \alpha_2 - 1) \ln (1 - x_1) + (\alpha_0 + \alpha_1-1) \ln (1 - x_2)\\
& \qquad + (\alpha_0 + \alpha_1 + \alpha_2) \left\{ \frac{(1 - x_1) \ln (1 - x_1)}{x_1} + \frac{(1 - x_2) \ln (1 - x_2)}{x_2} \right\} \\
& \qquad + (\alpha_0 + \alpha_1 + \alpha_2) \, \frac{\pi^2}{3} - 4
\end{align*}
and
\begin{multline*}
\clr (f_{\rm int})(x_1, x_2) =-(\alpha_0 + \alpha_1 + \alpha_2) \left\{ \frac{(1 - x_1) \ln (1 - x_1)}{x_1} + \frac{(1 - x_2) \ln (1 - x_2)}{x_2}\right\} \\
 -(\alpha_0 + \alpha_1 + \alpha_2 ) \left\{ \frac{\pi^2}{6} + \ln (1 - x_1x_2) \right\}.
\end{multline*}
It is readily seen that $\clr (f_{\rm int})(x_1, x_2) \to 0$ as $\alpha_0 + \alpha_1 + \alpha_2 \to 0$, as implied by Proposition~\ref{prop:3.3} (iii).
\end{example}

\section{Extension to the multivariate case\label{sec:4}}

In this section, representation \eqref{eq:3.4} of \cite{Hron/etal:2020} is extended to the multivariate case through an orthogonal decomposition of an arbitrary $d$-variate density $f_\mu \in \mathcal{B}^2 (\lambda)$ into a direct sum of $2^d-1$ mutually orthogonal parts, viz.
\begin{equation}
\label{eq:4.1}
f_\mu = \bigoplus_{I \subseteq D} g_{\mu, I}
\end{equation}
indexed by non-empty subsets $I \subseteq D = \{ 1, \ldots, d \}$, where each term $g_{\mu,I}$ depends only on the variables with indices in the set $I$. 

A multivariate extension of this sort was recently derived from first principles by \cite{Facevicova/etal:2022} in the discrete case. In contrast, the decomposition proposed here applies to any density $f_\mu \in \mathcal{B}^2 (\lambda)$ with finite reference measure $\lambda$. This is achieved with the help of the Hoeffding--Sobol decomposition formula due to \cite{Kuo/etal:2010}; see also \cite{Mercadier/etal:2022}. 

\subsection{Notation\label{sec:4.1}}

Let $(\Omega, \mathcal{A})$ be a $d$-dimensional product space and let $\lambda$ be a finite product reference measure. Specifically, suppose that for $i \in D$, $(\Omega_i, \mathcal{A}_i)$ is a measurable space and $\lambda_i$ is a finite, positive, real-valued measure on it. Set
\[
\Omega = \Omega_1 \times \cdots \times \Omega_d, \quad \mathcal{A} = \mathcal{A}_1 \otimes \cdots \otimes \mathcal{A}_d, \quad \lambda = \lambda_1 \otimes \cdots \otimes \lambda_d.
\]
As in Section~\ref{sec:3}, let $\mathcal{B}(\lambda)$ be the Bayes space on $(\Omega,\mathcal{A})$ and $\mathcal{B}^2 (\lambda)$ be the subspace of $\lambda$-densities whose logarithm is square-integrable.

For arbitrary non-empty $I \subseteq D$, let
\[
\Omega_I = \bigtimes_{i \in I} \Omega_i, \quad \mathcal{A}_I = \bigotimes_{i \in I} \mathcal{A}_i, \quad \lambda_I = \bigotimes_{i \in I} \lambda_i. 
\]
Further, let $\mathcal{B}(\lambda_I)$ denote the Bayes space on $(\Omega_I, \mathcal{A}_I)$ with reference measure $\lambda_I$, and let $\mathcal{B}^2(\lambda_I)$ be its subspace of $\lambda_I$-densities whose logarithm is square-integrable. By analogy with Eq.~\eqref{eq:2.4}, $L^2_0(\lambda_I)$ denotes the subspace of functions that integrate to $0$ of the $L_2$ space $L_2(\lambda_I)$ on $\Omega_I$ with reference measure $\lambda_I$.

Finally, for arbitrary $I \subseteq D$, $I \neq \varnothing$, let $\mathcal{B}^2_{ I }(\lambda)$ be the subspace of $\mathcal{B}^2 (\lambda)$ given by
\[
\mathcal{B}_I^2 (\lambda)=\Bigl\{\mu\in\mathcal{B}^2 (\lambda)\; \Bigl | \;\exists_{\mu_I \in\mathcal{B}^2(\lambda_I)} \forall_{A_i \in \mathcal{A}_i} \; \mu\Bigl(\bigtimes_{i \in D} A_i\Bigr)= \prod_{i \not \in I}\lambda_i(A_i) \mu_I \Bigl(\bigtimes_{i \in I} A_i \Bigr) \Bigr\},
\]
so that any $\lambda$-density in $\mathcal{B}_I^2 (\lambda)$ is a function of arguments with indices in $I$ only. Note that when $D = I$, $\mathcal{B}^2_D(\lambda)=\mathcal{B}^2 (\lambda)$. By convention, $\mathcal{B}^2_\varnothing(\lambda)$ is the trivial subspace that contains only the neutral element $\lambda$. 

\subsection{Higher-order geometric margins\label{sec:4.2}}

The following result is a multivariate analog of Lemmas~\ref{lem:3.2} and \ref{lem:3.4}. It plays a key role in the construction of the decomposition of multivariate densities $f_\mu \in \mathcal{B}^2 (\lambda)$. 

\begin{proposition}
\label{prop:4.1}
Consider an arbitrary $\lambda$-density $f_\mu \in \mathcal{B}^2 (\lambda)$ and an arbitrary non-empty set $I \subsetneq D$. Then  $\mathcal{B}_I^2 (\lambda)$ is a complete subspace of $\mathcal{B}^2 (\lambda)$ and the unique orthogonal projection of $f_\mu$ onto $\mathcal{B}_I^2 (\lambda)$ is given by
\[
f_{\mu, I} = \exp \left\{ \frac{1}{\lambda_{D \setminus I}(\Omega_{D \setminus I})}  \int_{\Omega_{D \setminus I}} \ln (f_\mu) d \lambda_{D \setminus I} \right\}.
\]
Furthermore, 
\[
\clr (f_{\mu, I}) = \frac{1}{\lambda_{D \setminus I}(\Omega_{D \setminus I})} \int_{\Omega_{D \setminus I}} \clr (f_\mu) d \lambda _{D \setminus I}.
\]
\end{proposition} 

\begin{proof}
To show that $\mathcal{B}_I^2 (\lambda)$ is complete, note that any $f_\mu \in \mathcal{B}_I^2 (\lambda)$ is a function of variables with indices in $I$ only, and so is its clr transformation. Hence, because the latter is an isomorphism, one has
\[
\| f_{\mu} \|^2_{\mathcal{B}^2 (\lambda)} = \| \clr (f_{\mu}) \|^2_{L^2 (\lambda)} = \lambda_{D \setminus I}(\Omega_{D \setminus I}) \| \clr (f_{\mu}) \|^2_{L^2(\lambda_I)} = \lambda_{D \setminus I}(\Omega_{D \setminus I}) \| f_{\mu} \|^2_{\mathcal{B}^2(\lambda_I)}. 
\]
The completeness of $\mathcal{B}_I^2 (\lambda)$ then follows from the completeness of $\mathcal{B}^2(\lambda_I)$. 

Proceeding along the same lines as in the proof of Proposition~\ref{prop:3.1}, choose an arbitrary $f_\mu \in \mathcal{B}^2 (\lambda)$. To prove that $f_{\mu, I}$  is the orthogonal projection of $f_\mu$ onto $\mathcal{B}_I^2 (\lambda)$, it suffices to check that, for any $g \in \mathcal{B}_I^2 (\lambda)$, one has
\[
\langle f_\mu \ominus f_{\mu, I}, g \rangle_{\mathcal{B}^2 (\lambda)} = \langle \clr (f_\mu) - \clr (f_{\mu, I}), \clr (g) \rangle_{L^2 (\lambda)} = 0.
\]
To verify this, one must first check that $\clr (f_{\mu, I})$ is as asserted. Indeed, it follows from Fubini's theorem that
\begin{multline*}
\int_\Omega \left\{ \frac{1}{\lambda_{D \setminus I}(\Omega_{D \setminus I})} \int_{\Omega_{D \setminus I}} \ln (f_\mu) d \lambda_{D \setminus I} \right\} d \lambda \\
= \int_{\Omega_I} \int_{\Omega_{D \setminus I}} \left\{ \frac{1}{\lambda_{D \setminus I}(\Omega_{D \setminus I})} \int_{\Omega_{D \setminus I}} \ln (f_\mu) d \lambda_{D \setminus I} \right\} d \lambda_{D \setminus I} \, d \lambda_I,
\end{multline*}
so that
\begin{multline*}
\clr (f_{\mu, I}) = \frac{1}{\lambda_{D \setminus I}(\Omega_{D \setminus I})} \int_{\Omega_{D \setminus I}} \ln (f_\mu) d \lambda_{D \setminus I} \\
- \frac{1}{\lambda(\Omega)} \int_\Omega \left\{ \frac{1}{\lambda_{D \setminus I}(\Omega_{D \setminus I})} \int_{\Omega_{D \setminus I}} \ln (f_\mu) d \lambda_{D \setminus I} \right\} d \lambda
\end{multline*}
can be rewritten as
\begin{align*}
\clr (f_{\mu, I}) & = \frac{1}{\lambda_{D \setminus I}(\Omega_{D \setminus I})} \int_{\Omega_{D \setminus I}} \ln (f_\mu) d \lambda_{D \setminus I}  - \frac{1}{\lambda(\Omega)} \int_\Omega  \ln (f_\mu) d \lambda\\ 
& = \frac{1}{\lambda_{D \setminus I}(\Omega_{D \setminus I})} \int_{\Omega_{D \setminus I}} \clr ( f_\mu) \, d \lambda_{D \setminus I} ,
\end{align*}
as claimed. Therefore,
\begin{align*}
\langle \clr (f_\mu) - \clr (f_{\mu, I}), \clr (g) \rangle_{L^2 (\lambda)} & = \int_{\Omega_I} \clr (g)  \int_{\Omega_{D \setminus I}} \{\clr (f_\mu) - \clr (f_{\mu, I})\}  \, d \lambda_{D \setminus I} \, d \lambda_I \\
 & = \int_{\Omega_I} \clr (g)  \lambda_{D \setminus I}(\Omega_{D \setminus I})\{ \clr (f_{\mu, I} - \clr (f_{\mu, I})\} \, d \lambda_I.
\end{align*}
The latter expression is clearly equal to $0$, as was to be shown. 
\end{proof}

Note that when $I = \varnothing$, one can define, analogously to Proposition~\ref{prop:4.1}, 
\begin{equation}
\label{eq:4.2}
f_{\mu, \varnothing} =  \exp \left\{ \frac{1}{\lambda(\Omega)}  \int_{\Omega} \ln (f_\mu) d \lambda \right\} \prop2 1.
\end{equation}
Furthermore, when $d = 2$ and $I = \{i\}$ for some $i \in \{1, 2 \}$, the projection $f_{\mu, I} = f_{\mu, \{ i \}}$ defined in Proposition~\ref{prop:4.1} is precisely the geometric margin $f_{\mu, i}$ defined in Section~\ref{sec:3}. This remark motivates the introduction of the following definition.

\begin{definition}
\label{def:4.1}
For arbitrary non-empty set $I \subsetneq D$, $f_{\mu, I}$ is called the $I$-th geometric margin of $f_\mu$. When $I = \{ i \}$ for some integer $i \in D$, one writes $f_{\mu, i}$ instead of $f_{\mu, \{i\}}$.  Furthermore, by convention, $f_{\mu, D}=f_\mu$ and $f_{\mu,\varnothing}$ is given by \eqref{eq:4.2}.
\end{definition}

\subsection{Decomposition of the multivariate density\label{sec:4.3}}

In order to call on Theorem~2.1 of \cite{Kuo/etal:2010}, introduce, for any $I \subseteq D$, the operator $P_{D \setminus I} : \mathcal{B}^2 (\lambda) \to \mathcal{B}^2_{ I}(\lambda)$ defined, for any $f \in \mathcal{B}^2 (\lambda)$, by
\begin{equation}
\label{eq:4.3}
P_{D \setminus I}(f_\mu) = f_{\mu, I}
\end{equation}
with $f_{\mu, I}$ as in Definition~\ref{def:4.1}. In particular, $P_\varnothing = {\rm id}$ is the identity map $f_\mu \mapsto f_\mu$. When $I = D \setminus \{ i \}$ for some $i \in D$, one has $D \setminus (D \setminus I) = \{i\}$ and simply writes $P_i \equiv P_{\{i\}}$.

With these conventions, which are chosen to match the notation of \cite{Kuo/etal:2010}, the form of the $I$-th geometric margin $f_{\mu, I}$ as given in Definition~\ref{def:4.1} implies that, for any $I \subseteq D$, 
\[
P_I = \prod_{i \in I} P_i, 
\]
where the product notation refers to convolution of mappings and an empty product refers to $P_\varnothing = {\rm id}$.  Furthermore, the following properties of $P_i$ hold. 

\begin{lemma}
\label{lem:4.1}
The operators $P_1, \ldots, P_d$ are commuting projections on $\mathcal{B}^2 (\lambda)$ such that, for all $f_\mu \in \mathcal{B}^2 (\lambda)$, $P_i(f_\mu) \in  \mathcal{B}^2_{D \setminus \{i\}}(\lambda)$ and $P_i (f_\mu) = f_\mu$ whenever $f_\mu \in \mathcal{B}^2_{D \setminus \{i\}}(\lambda)$.
\end{lemma}

\begin{proof}
To show that the operators are commuting projections, one must check that, for all $i, j \in D$, $P_i \circ P_j = P_j \circ P_i$ and $P_i$ is  linear and idempotent, the latter meaning that $P_i \circ P_i = P_i$. The linearity of $P_i$ is apparent from Proposition~\ref{prop:4.1}, and so is the fact that $P_i \circ P_j = P_{\{i,j\}} = P_j \circ P_i$. Given that $P_i$ is a projection onto $\mathcal{B}^2_{D \setminus \{i\}}(\lambda)$, it is also immediate from Proposition~\ref{prop:4.1} that, for all $f_\mu \in \mathcal{B}^2 (\lambda)$, one has $P_i(f_\mu) \in  \mathcal{B}^2_{D \setminus \{i\}}(\lambda)$ and $P_i (f_\mu) = f_\mu$ whenever $f_\mu \in \mathcal{B}^2_{D \setminus \{i\}}(\lambda)$. That $P_i$ is idempotent is a trivial consequence of this observation. This concludes the argument.
\end{proof}

Lemma~\ref{lem:4.1} implies that the assumptions of Theorem~2.1 of \cite{Kuo/etal:2010} hold. The result below then follows directly from Eqs.~(2.6) and (2.7) in that paper. 

\begin{theorem}
\label{thm:4.1}
Any $f_\mu \in \mathcal{B}^2 (\lambda)$ can be decomposed in the form~\eqref{eq:4.1}, where for any $I \subseteq D$, $g_{\mu, I}$ is given by
\[
g_{\mu, I} = \Bigl ( \prod_{i \in I} ({\rm id} - P_i)\Bigr) P_{D \setminus I} (f _\mu).
\]
\end{theorem}

The next result explains how the decomposition in Theorem~\ref{thm:4.1} relates to the one of \cite{Hron/etal:2020} in Eq.~\eqref{eq:3.4} in the bivariate case. 

\begin{proposition}
\label{prop:4.2}
For any $f_\mu \in \mathcal{B}^2 (\lambda)$, 
\[
f_\mu  =  f_{\mu, {\rm ind}} \oplus \bigoplus_{I \subseteq D, | I | \ge 2} f_{\mu, I, {\rm int}} 
\]
where the so-called independence and interaction parts are respectively given by
\[
 f_{\mu, {\rm ind}} = \bigoplus_{i=1}^d f_{\mu, i}, \quad f_{\mu, I, {\rm int}}  =  \bigoplus_{J \subseteq I, J \neq \varnothing}  \{ (-1)^{|I \setminus J|}\} \odot f_{\mu, J}.
\]
\end{proposition}

\begin{proof} To prove the claim, one must relate the functions $g_{\mu, I}$ in Theorem~\ref{thm:4.1} to the elements of the decomposition in Proposition~\ref{prop:4.2}. To this end, set $\kappa = P_D (f_\mu) = f_{\mu, \varnothing}$ and keep in mind that $\kappa \in \mathbb{R}$ is a constant. From the definition of $P$ in Eq.~\eqref{eq:4.3} and  Theorem~2.1 (b) in  \cite{Kuo/etal:2010}, one also has, for any $I \subseteq D$,
\[
g_{\mu, I} = \bigoplus_{J \subseteq I} \{ (-1)^{|I \setminus J|}\} \odot P_{D \setminus J}(f_\mu) = [\{(-1)^{|I|}\}\odot \kappa] \oplus  \bigoplus_{J \subseteq I, J \neq \varnothing}  \{ (-1)^{|I \setminus J|}\} \odot f_{\mu, J}.
\]
In particular, $g_{\mu,\varnothing}=\kappa$ and if $I = \{ i \}$ for some integer $i \in D$, then $g_{\mu, I} =f_{\mu, i}/\kappa$. The decomposition in Theorem~\ref{thm:4.1} can then be rewritten as 
\[
f_\mu = \bigoplus_{I \subseteq D} [\{(-1)^{|I|}\}\odot \kappa] \bigoplus_{I \subseteq D, I \neq \varnothing} \bigoplus_{J \subseteq I, J \neq \varnothing}  \{ (-1)^{|I \setminus J|}\} \odot f_{\mu, J}.
\]
Given that 
\[
\bigoplus_{I \subseteq D} [\{(-1)^{|I|}\}\odot \kappa] = \kappa^{\sum_{I \subseteq D} (-1)^{|I|}} = \kappa^{(1-1)^{|D|}} = 1,
\]
one has 
\[
f_\mu = \bigoplus_{I \subseteq D, I \neq \varnothing} \bigoplus_{J \subseteq I, J \neq \varnothing}  \{ (-1)^{|I \setminus J|}\} \odot f_{\mu, J}
=\bigoplus_{i=1}^d f_{\mu, i} \oplus \bigoplus_{I \subseteq D, |I |\ge 2} \bigoplus_{J \subseteq I, J \neq \varnothing}  \{ (-1)^{|I \setminus J|}\} \odot f_{\mu, J}
\]
which is easily seen to have the desired form upon defining the independence and interaction parts as in Proposition~\ref{prop:4.2}.
\end{proof}

When $d = 2$, Proposition~\ref{prop:4.2} boils down to the decomposition~\eqref{eq:3.4}, with $f_{\mu, D, {\rm int}}$ being the same as $f_{\mu, {\rm int}}$ in Eq~\eqref{eq:3.2}. In general when $d \ge 2$ there are $2^d - (d+1)$ interaction parts indexed by sets $I \subseteq D$ with $|I| \ge 2$ and termed $I$-th interactions. 

Interestingly, each $I$-th interaction with $|I| \ge 2$ can be computed recursively from the univariate geometric margins and interactions pertaining to subsets $J \subsetneq I$. As seen below, this is a consequence of Theorem~4.2 (a) in \cite{Kuo/etal:2010}. 

\begin{proposition}
\label{prop:4.3}
For any $f_\mu \in \mathcal{B}^2 (\lambda)$ and $I \subseteq D$ with $|I| \ge 2$, the $I$-th interaction part in Proposition~\ref{prop:4.2} satisfies
\[
f_{\mu, I, {\rm int}} = f_{\mu, I} \ominus \Bigl[ \bigoplus_{J \subsetneq I, |J| \ge 2} f_{\mu, J, {\rm int}} \bigoplus_{i\in I} f_{\mu,i} \Bigr].
\]
\end{proposition} 

\begin{proof}
Fix an arbitrary $f_\mu \in \mathcal{B}^2 (\lambda)$ and a subset $I \subseteq D$ with $|I| \ge 2$. From the proof of Proposition~\ref{prop:4.2}, it transpires that 
\[
f_{\mu, I, {\rm int} } = g_{\mu, I} \ominus [ \{(-1)^{|I|}\} \odot \kappa ] = g_{\mu, I} \oplus [ \{(-1)^{|I|+1}\} \odot \kappa ] = g_{\mu, I} \oplus \kappa^{(-1)^{|I|+1}},
\]
so that $g_{\mu, I} = \kappa^{(-1)^{|I|}}  \oplus f_{\mu, I, {\rm int} }$. Theorem~2.1~(a) in~\cite{Kuo/etal:2010} then implies
\[
f_{\mu, I, {\rm int} } = \kappa^{(-1)^{|I|+1}} \oplus g_{\mu, I} = \kappa^{(-1)^{|I|+1}} \oplus \Bigl[ f_{\mu, I} \ominus \bigoplus_{J \subsetneq I} g_{\mu, J} \Bigr]. 
\]
Now observe that
\[
\bigoplus_{J \subsetneq I} g_{\mu, J}  = \bigoplus_{J \subseteq I, |J| \ge 2 }\{ f_{\mu, J, {\rm int}} \oplus \kappa^{|J|} \} \bigoplus_{i \in I} \{ f_{\mu, i} \oplus \kappa^{-1}\} \oplus \kappa .
\]
Because $\oplus$ is commutative, the terms involving $\kappa$ can be factored out, leading to 
\[
 \kappa^{\sum_{J \subsetneq I} (-1)^{|J|}} = \kappa^{(1-1)^{|I|} - (-1)^{|I|}}= \kappa^{(-1)^{|I|+1}}  
 \]
by the multinomial formula. Accordingly,
\[
f_{\mu, I, {\rm int} } =\{ \kappa^{(-1)^{|I|+1}} \ominus \kappa^{(-1)^{|I|+1}}\} \oplus f_{\mu, I} \ominus \Bigl[ \bigoplus_{J \subseteq I, |J| \ge 2 } f_{\mu, J, {\rm int}}  \bigoplus_{i \in I} f_{\mu, i} \Bigr] ,
\]
which is precisely the claimed expression, as $\kappa^{(-1)^{|I|+1}} \ominus \kappa^{(-1)^{|I|+1}}=1$.
\end{proof}

\subsection{Orthogonality of the components in the decomposition\label{sec:4.4}}

In the bivariate case, Proposition~\ref{prop:3.2} and Corollary~\ref{cor:3.2} established that the elements of the decomposition of $f_\mu \in \mathcal{B}^2 (\lambda)$ are orthogonal projections on certain orthogonal subspaces of $\mathcal{B}^2 (\lambda)$.  This then implies that the independence and interaction part of the decomposition are orthogonal to each other, which had a number of useful consequences. It will now be seen that an analog exists in the multivariate case. To this end, the following two lemmas are instrumental. 

\begin{lemma}
\label{lem:4.2}
If $I$ and $J$ are two non-empty subsets of $D$, then the corresponding subspaces $\mathcal{B}_I^2 (\lambda)$ and $\mathcal{B}_J^2 (\lambda)$ are orthogonal.
\end{lemma}

\begin{proof}
Choose arbitrary $f_\mu \in \mathcal{B}^2_{ I }(\lambda)$ and $g_\nu \in \mathcal{B}^2_J(\lambda)$. Because the clr transformation is an isometry, it suffices to prove that $\langle \clr (f_\mu), \clr (g_\nu) \rangle_{L^2 (\lambda)} = 0$. This is follows at once from the fact that $\clr (f_\mu)$ and $\clr (g_\nu)$ are functions of variables in $\Omega_I$ and $\Omega_J$, respectively. Indeed,
\[
\langle \clr (f_\mu), \clr (g_\nu) \rangle_{L^2 (\lambda)} = \int_{\Omega_{D \setminus (I \cup J)}} \int_{\Omega_J} \clr (g_\nu) \Bigl\{ \int_{\Omega_I} \clr (f_\mu) d \lambda_I \Bigr\} d \lambda_J d \lambda_{D\setminus (I\cup J)} =0,
\] 
where the last equality follows from the fact that the term in curly brackets is $0$ by construction of the clr transformation.
\end{proof}

Now consider a non-empty set $I \subseteq D$ and a non-empty subset thereof, say $K \subsetneq I$. Because $\mathcal{B}^2_K(\lambda) \subsetneq \mathcal{B}^2_{ I}(\lambda)$, one can consider the orthogonal complement of $\mathcal{B}^2_K(\lambda)$ in $\mathcal{B}^2_{ I }(\lambda)$, viz.
\begin{equation}
\label{eq:4.4}
\mathcal{B}_{K,I, \perp} = \{ f \in \mathcal{B}^2_{ I }(\lambda) \; | \; \forall_{g \in \mathcal{B}^2_K(\lambda)} \; \langle f, g \rangle_{\mathcal{B}^2 (\lambda)} =0\}.
\end{equation}

\begin{lemma}
\label{lem:4.3}
Consider arbitrary distinct non-empty subsets $I, J \subseteq D$ such that $I \cap J = K \neq \varnothing$. Then $\mathcal{B}^2_{K, I,\perp}$ and $\mathcal{B}^2_{K, J,\perp}$ are orthogonal.
\end{lemma}

\begin{proof}
Without loss of generality, assume that $I \setminus K \neq \varnothing$. First recall that any $f_\mu \in \mathcal{B}^2_{ I }(\lambda)$ is a function of variables in $\Omega_I$ only. From Proposition~\ref{prop:4.1}, its unique orthogonal projection on $\mathcal{B}^2_K$ satisfies 
\[
\ln (f_{\mu, K})= \frac{1}{\lambda_{D \setminus K}(\Omega_{D \setminus K})}  \int_{\Omega_{D \setminus K}} \ln (f_\mu) d \lambda_{D \setminus K} =  \frac{1}{\lambda_{I \setminus K}(\Omega_{I \setminus K})}  \int_{\Omega_{I \setminus K}} \ln (f_\mu) d \lambda_{I \setminus K}.
\]
Similarly,
\begin{equation}\label{eq:4.5}
\clr (f_{\mu, K}) = \frac{1}{\lambda_{I \setminus K}(\Omega_{I \setminus K})} \int_{\Omega_{I \setminus K}} \clr (f_\mu) d \lambda _{I \setminus K}.
\end{equation}

Now take arbitrary elements $f \in \mathcal{B}^2_{K,I, \perp} (\lambda)$ and $g \in \mathcal{B}^2_{K,J,\perp}(\lambda)$. Then there exist $f_\mu \in \mathcal{B}^2_{ I }(\lambda)$ and $g_\nu \in \mathcal{B}^2_J(\lambda)$ so that $f = f_\mu \ominus f_{\mu, K}$ and $g = g_\nu \ominus g_{\nu, K}$. Hence
\[ 
\langle \clr (f), \clr (g) \rangle_{L^2 (\lambda)} = \langle \clr (f_\mu)-\clr (f_{\mu,K}), \clr (g_\nu)-\clr (g_{\nu,K}) \rangle_{L^2 (\lambda)}.
\]
The latter expression can be written as
\[
\int_{\Omega_{D \setminus (I \cup J)}} \int_{\Omega_{J}} \{\clr (g_\nu)-\clr (g_{\nu,K})\} \Bigl[\int_{\Omega_{I \setminus K}} \{\clr (f_\mu)-\clr (f_{\mu,K})\} d \lambda_{I \setminus K}\Bigr] d \lambda_J d \lambda_{D \setminus (I \cup J)}.
\]
In view of Eq.~\eqref{eq:4.5}, the term in square brackets vanishes and hence $\langle f, g\rangle_{\mathcal{B}^2 (\lambda)} = 0$, as claimed.
\end{proof}

To state the main result of this section, the spaces $\mathcal{B}_{I,\rm int}^2 (\lambda)$ must be introduced, and this for each set $I \subseteq D$ with $|I| \ge 2$. This has to be done recursively. The first step is detailed in the following lemma.

\begin{lemma}\label{lem:4.4}
For each set $I = \{i, j\}$ with arbitrary $i \neq j \in D$, let $\mathcal{B}_{I,\rm int}^2 (\lambda)$ be the orthogonal complement of the direct sum $\mathcal{B}^2_{\{i\}}(\lambda) \oplus \mathcal{B}^2_{\{j\}}(\lambda)$ in $\mathcal{B}^2_{I}(\lambda)$, viz.
\[
\mathcal{B}_{I,\rm int}^2 (\lambda) = \{ f_\mu \in \mathcal{B}^2_{ I}(\lambda) \; | \; \forall_{f_{\nu} \in  \mathcal{B}^2_{\{i\}}(\lambda) \oplus \mathcal{B}^2_{\{j\}}(\lambda)} \;  \langle f_\mu , f_\nu \rangle_{\mathcal{B}^2 (\lambda)} = 0 \}.
\]
The subspaces $\mathcal{B}^2_{\{1\}}(\lambda), \ldots, \mathcal{B}^2_{\{d\}}(\lambda)$ and $\mathcal{B}_{I,\rm int}^2 (\lambda)$ with any $I \subseteq D$ with $|I| =2$ are then mutually orthogonal. 
\end{lemma}

\begin{proof}
The fact that the subspaces $\mathcal{B}^2_{\{ 1 \}} (\lambda), \ldots , \mathcal{B}^2_{\{d\}} (\lambda)$ are mutually orthogonal is already clear from Lemma~\ref{lem:4.2}. The same result also implies that $\mathcal{B}_{I,\rm int}^2 (\lambda)$ and $\mathcal{B}_{J,\rm int}^2 (\lambda)$ are orthogonal whenever $I \cap J = \varnothing$ and that $\mathcal{B}_{I,\rm int}^2 (\lambda)$ and $\mathcal{B}_{ \{ i \}}^2 (\lambda)$ are orthogonal whenever $i \not \in I$. It also follows by construction that $\mathcal{B}_{I,\rm int}^2 (\lambda)$ and $\mathcal{B}_{\{ i \}}^2 (\lambda)$ are orthogonal whenever $i \in I$. It thus remains to show that  $\mathcal{B}_{I,\rm int}^2 (\lambda)$ and $\mathcal{B}_{J,\rm int}^2 (\lambda)$ are orthogonal whenever $I \cap J \neq \varnothing$. Given that $|I| = |J|$, one has $I \cap J = K$ with $K = \{k\}$ for some $k \in D$, and hence $\mathcal{B}_{I,\rm int}^2 (\lambda) \subseteq \mathcal{B}_{K,I,\perp}^2 (\lambda)$, where $\mathcal{B}_{K,I,\perp}^2 (\lambda)$ is as in Eq.~\eqref{eq:4.4}. Similarly, one has $\mathcal{B}_{J,\rm int}^2 (\lambda) \subseteq \mathcal{B}_{K,J,\perp}^2 (\lambda)$. The orthogonality of $\mathcal{B}_{I,\rm int}^2 (\lambda)$ and $\mathcal{B}_{J,\rm int}^2 (\lambda)$ then follows at once from Lemma~\ref{lem:4.3}.
\end{proof}

If $d=2$, then there is nothing more to do; $I = D = \{1,2\}$ in Lemma~\ref{lem:4.4} and $\mathcal{B}_{D,\rm int}^2 (\lambda)$ then corresponds to the interaction space $\mathcal{B}^2_{\rm int}(\lambda)$ in Eq.~\eqref{eq:3.3} in Section~\ref{sec:3.5}.

If $d >2$, the construction of the interaction subspaces proceeds by induction, in which Lemma~\ref{lem:4.4} constitutes the base step. The induction hypothesis assumes that $\mathcal{B}_{I,\rm int}^2 (\lambda)$ has been defined for all subsets of cardinality at most $k$, where $2 \le k < d$ and that the subspaces $\mathcal{B}^2_{\{1\}}(\lambda), \ldots, \mathcal{B}^2_{\{d\}}(\lambda)$, and $\mathcal{B}_{I,\rm int}^2 (\lambda) $ with $I \subseteq D$, $2 \le |I| \le k$ are mutually orthogonal. 
For arbitrary $I \subseteq D$ with $|I| = k+1$, $\mathcal{B}_{I,\rm int}^2 (\lambda)$ is then defined as the orthogonal complement of the direct sum 
\begin{equation}\label{eq:4.6}
\mathcal{B}_{I,\oplus}^2 (\lambda)  = \bigoplus_{i\in I} \mathcal{B}^2_{\{i\}}(\lambda) \bigoplus_{J \subsetneq I, |J| \ge 2} \mathcal{B}_{J,\rm int}^2 (\lambda)
\end{equation}
in $\mathcal{B}^2_{I}(\lambda)$, i.e.,
\begin{equation}\label{eq:4.7}
\mathcal{B}_{I,\rm int}^2 (\lambda) = \{ f_\mu \in \mathcal{B}^2_{ I}(\lambda) \; | \; \forall_{f_{\nu} \in  \mathcal{B}_{I,\oplus}^2 (\lambda)} \;  \langle f_\mu , f_\nu \rangle_{\mathcal{B}^2 (\lambda)} = 0 \}.
\end{equation}
Note that the orthogonality of the subspaces in the induction hypothesis guarantees that the definitions \eqref{eq:4.6} and \eqref{eq:4.7} are meaningful.  Therefore, for the induction to proceed, this orthogonality needs to be extended to subspaces featuring subsets $I$ or cardinality at most $k+1$. This is done next.

\begin{lemma}
\label{lem:4.5}
Under the induction hypothesis, the subspaces $\mathcal{B}^2_{\{1\}}(\lambda), \ldots, \mathcal{B}^2_{\{d\}}(\lambda)$, and $\mathcal{B}_{I,\rm int}^2 (\lambda) $ with $I \subseteq D$, $2 \le |I| \le k+1$ are mutually orthogonal.
\end{lemma}

\begin{proof}
Considering the induction hypothesis, it remains to be shown that for an arbitrary $I \subseteq D$ with $|I| = k+1$,  $\mathcal{B}_{I,\rm int}^2 (\lambda)$ is orthogonal to (i) $\mathcal{B}_{\{i\}}^2 (\lambda)$ for any $i \in D$; and (ii) any $\mathcal{B}_{J,\rm int}^2 (\lambda)$ where $J \subseteq D$, $I \neq J$ and $|J| \le k+1$. As in the proof of Lemma~\ref{lem:4.4}, (i) follows by construction if $i \in I$ and from Lemma~\ref{lem:4.2} when $i \neq I$. To show (ii), first observe that the definition \eqref{eq:4.7} implies the orthogonality of $\mathcal{B}_{I,\rm int}^2 (\lambda)$ and $\mathcal{B}_{J,\rm int}^2 (\lambda)$ whenever $J \subsetneq I$, because $\mathcal{B}_{J,\rm int}^2 (\lambda) \subseteq \mathcal{B}_{I,\oplus}^2 (\lambda)$. Also, because $\mathcal{B}_{I,\rm int}^2 (\lambda) \subseteq \mathcal{B}_{I}^2 (\lambda)$ and similarly for $J$, (ii) follows from Lemma~\ref{lem:4.2} whenever $I \cap J = \varnothing$. It thus remains to prove (ii) when $I \cap J = K$ where $K \neq\varnothing$ and $K \neq J$.  Here one can again proceed similarly as in the proof of Lemma~\ref{lem:4.4}. Namely, it is easy to see that $\mathcal{B}_{I,\rm int}^2 (\lambda) \subseteq \mathcal{B}_{K,I, \perp}(\lambda) $ and $\mathcal{B}_{J,\rm int}^2 (\lambda) \subseteq \mathcal{B}_{K,J, \perp}(\lambda) $. This implies (ii) because $\mathcal{B}_{K,I, \perp}(\lambda) $ and $\mathcal{B}_{K,J, \perp}(\lambda)$ are orthogonal by Lemma~\ref{lem:4.3}.
\end{proof}

It is now possible to state and prove the main result of this section.

\begin{theorem}
\label{thm:4.2}
For any $f_\mu \in \mathcal{B}^2 (\lambda)$ and any $I \subseteq D$ with $|I| \ge 2$, $f_{\mu, I, {\rm int}}$ is the unique orthogonal projection of $f_\mu$ on $\mathcal{B}_{I,\rm int}^2 (\lambda)$. Furthermore,
\[
\mathcal{B}^2 (\lambda) = \bigoplus_{i\in I} \mathcal{B}^2_{ \{ i \} }(\lambda) \bigoplus_{I \subsetneq D, |I| \ge 2} \mathcal{B}_{I,\rm int}^2 (\lambda).
\]
 and all spaces featuring in the direct sum are mutually orthogonal.
\end{theorem}

\begin{proof}
First note that for any $I \subseteq D$ with $|I| \ge 2$, $\mathcal{B}_{I,\rm int}^2 (\lambda)$ is a complete subspace of $\mathcal{B}^2_{ I }(\lambda)$ by construction and Proposition~\ref{prop:4.1}. The claim that for any $I \subseteq D$ with $|I| \ge 2$, $f_{\mu, I, {\rm int}}$ is the unique orthogonal projection of $f_\mu$ on $\mathcal{B}_{I,\rm int}^2 (\lambda)$ can be showed by induction. When $I = \{i, j\}$ for some distinct $i, j \in D$, the base step is a consequence of Lemma~\ref{lem:4.4}, given the fact that $f_{\mu, I, {\rm int}} = f_{\mu, I} \ominus f_{\mu,i} \ominus f_{\mu,j}$ as established in Proposition~\ref{prop:4.3}, and that $f_{\mu,i}$ and $f_{\mu,j} $ are orthogonal projections on $\mathcal{B}^2_{\{i\}}(\lambda)$ and  $\mathcal{B}^2_{\{j\}}(\lambda)$, respectively, stated in Proposition~\ref{prop:4.1}. 

To carry out the inductive step, suppose that for any $J \subsetneq D$ with $2 \le |J| \le k$, $f_{\mu, J, {\rm int}}$ is the unique orthogonal projection of $f_\mu$ on $\mathcal{B}^2_{J, \rm int}(\lambda)$. Let $I \subseteq D$ be arbitrary with $|I| =k+1$. The induction hypothesis implies that 
\[
\bigoplus_{i\in I} f_{\mu,i} \oplus\bigoplus_{J \subsetneq I, |J| \ge 2} f_{\mu, J, {\rm int}} 
\]
is the unique orthogonal projection of $f_{\mu}$ on $\mathcal{B}^2_{I,\oplus}(\lambda)$. Hence Eq.~\eqref{eq:4.6} and the expression of $f_{\mu, I, \rm int}$ in  Proposition~\ref{prop:4.3} together imply that $f_{\mu, I, \rm int}$ is the unique orthogonal projection of $f_{\mu}$ on $\mathcal{B}^2_{I,\rm int}(\lambda)$. 

Finally, the mutual orthogonality of the spaces $\mathcal{B}^2_{\{1\}}(\lambda), \ldots, \mathcal{B}^2_{\{d\}}(\lambda)$, and $\mathcal{B}_{I,\rm int}^2 (\lambda) $ with $I \subseteq D$, $2 \le |I|$ follows easily by induction from Lemma~\ref{lem:4.4} and Lemma~\ref{lem:4.5}.
\end{proof}

By analogy to Eq.~\eqref{eq:3.3} in the bivariate case, the so-called independence space as 
\[
\mathcal{B}_{\rm ind}^2 (\lambda) = \{ \mu \in \mathcal{B}^2 (\lambda) \;| \;  \exists_{\mu_1 \in \mathcal{B}^2(\lambda_1)}\; \cdots \; \exists_{\mu_d \in \mathcal{B}^2(\lambda_d)} \; \mu = \mu_1 \otimes \cdots \otimes \mu_d\}.
\]
Using analogous arguments as in the proof of Lemma~\ref{lem:3.3}, one finds that $\mathcal{B}_{\rm ind}^2 (\lambda)$ is complete and
\[
\mathcal{B}_{\rm ind}^2 (\lambda) = \mathcal{B}_{\{1\}}^2 (\lambda) \oplus \cdots \oplus  \mathcal{B}_{\{d\}}^2 (\lambda).
\]
The following final result of this subsection, which follows directly from Theorem~\ref{thm:4.2}, is the multivariate analog of Proposition~\ref{prop:3.2}.

\begin{corollary}
For any $f_\mu \in \mathcal{B}^2 (\lambda)$,  the independence part $f_{\mu, {\rm ind}}$  introduced in Proposition~\ref{prop:4.2} is the unique orthogonal projection of $f_\mu$ on $\mathcal{B}_{\rm ind}^2 (\lambda)$.
\end{corollary}

\subsection{Illustration\label{sec:4.5}}

To form a mental image of the results developed in this section, it may be useful to give a visual representation of the decomposition of the Bayes space $\mathcal{B}^2 (\lambda)$ in the 3-dimensional case. This is done in the example below. A concrete example of orthogonal decomposition of a $d$-variate density will then be presented in Section~\ref{sec:5} for copulas, where the reference measure $\lambda$ is Lebesgue measure on $[0, 1]^d$.

\begin{example}
\label{ex:4.1}
Figure~\ref{fig:1} displays a Venn diagram of an abstract 3-dimensional Bayes space $\mathcal{B}^2 (\lambda)$. In the picture, the sectors marked $1$, $2$, $3$ correspond to the subspaces $\mathcal{B}^2_{\{ 1 \}} (\lambda)$, $\mathcal{B}^2_{\{ 2 \}} (\lambda)$, $\mathcal{B}^2_{\{ 3 \}} (\lambda)$ of maps $f_\mu \in \mathcal{B}^2 (\lambda)$ of three variables that are respectively functions of the first, second, and third argument only. 

The orthogonality of the subspaces $\mathcal{B}^2_{\{ 1 \}} (\lambda)$, $\mathcal{B}^2_{\{ 2 \}} (\lambda)$, $\mathcal{B}^2_{\{ 3 \}} (\lambda)$, established in Lemma~\ref{lem:4.2}, is conveyed graphically by the fact that they do not intersect (though they do in reality, if only because the neutral element $1$ is common to them all). 

In the $3$-variate orthogonal decomposition of $f_\mu$, the components $f_{\mu,1}$, $f_{\mu,2}$, and $f_{\mu,3}$ are the orthogonal projections of $f_\mu$ onto these subspaces, and their direct sum, namely $f_{\mu, \rm ind} = f_{\mu,1} \oplus f_{\mu,2} \oplus f_{\mu,3}$ belongs to the subspace
\[
\mathcal{B}^2_{\rm ind} (\lambda) = \mathcal{B}^2_{\{ 1 \}} (\lambda) \oplus \mathcal{B}^2_{\{ 2 \}} (\lambda) \oplus \mathcal{B}^2_{\{ 3 \}} (\lambda),
\]
which is represented in dark grey in the figure. Therefore, the independent part of the decomposition lies in that dark grey zone.

Represented in light grey in the same graph and respectively labeled $12$, $13$, $23$ are the three orthogonal subspaces $\mathcal{B}^2_{I,{\rm int}} (\lambda)$ corresponding to the subsets $I = \{ 1, 2 \}$, $\{ 1, 3 \}$, and $\{ 2, 3 \}$. Their lack of intersection among themselves and with the sets labeled $1$, $2$, $3$ again conveys their mutual orthogonality established in Lemma~\ref{lem:4.3}.

In the $3$-variate orthogonal decomposition of $f_\mu$, the components $f_{\mu, \{ 1, 2 \}}$, $f_{\mu, \{ 1, 3 \}}$, and $f_{\mu, \{ 2, 3 \}}$ are the orthogonal projections of $f_\mu$ onto these subspaces, and their direct sum, namely $f_{\mu, \{ 1, 2 \}, {\rm int}} \oplus f_{\mu, \{ 1, 3 \}, {\rm int}} \oplus f_{\mu, \{ 2, 3 \}, {\rm int}}$ belongs to the subspace
\[
\mathcal{B}^2_{ \{ 1, 2 \}, {\rm int}} (\lambda) \oplus \mathcal{B}^2_{ \{ 1, 3 \}, {\rm int}} (\lambda) \oplus \mathcal{B}^2_{ \{ 2, 3 \}, {\rm int}} (\lambda),
\]
which is represented in light grey in the figure.

Finally, the white annulus labeled $123$ in Figure~\ref{fig:1} is $\mathcal{B}^2_{\{1,2,3\}, {\rm int}} (\lambda)$, i.e., the orthogonal complement of the direct sum $\mathcal{B}^2_{ \{ 1, 2,3 \}, \oplus} (\lambda)$ of the grey zones defined in Eq.~\eqref{eq:4.6}. Summing up, one has 
\begin{multline*}
\mathcal{B}^2 (\lambda) = \mathcal{B}^2_{\{ 1 \}} (\lambda) \oplus \mathcal{B}^2_{\{ 2 \}} (\lambda) \oplus \mathcal{B}^2_{\{ 3 \}} (\lambda) \\
\oplus \mathcal{B}^2_{\{1, 2 \}, {\rm int}} (\lambda) 
\oplus \mathcal{B}^2_{\{1, 3 \}, {\rm int}} (\lambda) \oplus \mathcal{B}^2_{\{2, 3 \}, {\rm int}} (\lambda) \oplus \mathcal{B}^2_{\{1, 2, 3 \}, {\rm int}} (\lambda),
\end{multline*}
as guaranteed by Theorem~\ref{thm:4.2}, and the orthogonal decomposition of 	any $f_\mu \in \mathcal{B}^2 (\lambda)$ takes the form
\[
f_\mu = f_{\mu, 1} \oplus f_{\mu, 2} \oplus f_{\mu, 3} \oplus f_{\mu, \{ 1, 2 \},\rm int} \oplus  f_{\mu, \{ 1, 3 \},\rm int} \oplus  f_{\mu, \{ 2, 3 \},\rm int} \oplus  f_{\mu, \{ 1, 2, 3 \},\rm int}.
\]
\end{example}

\begin{figure}[t!]
\begin{center}
\includegraphics[height=8cm]{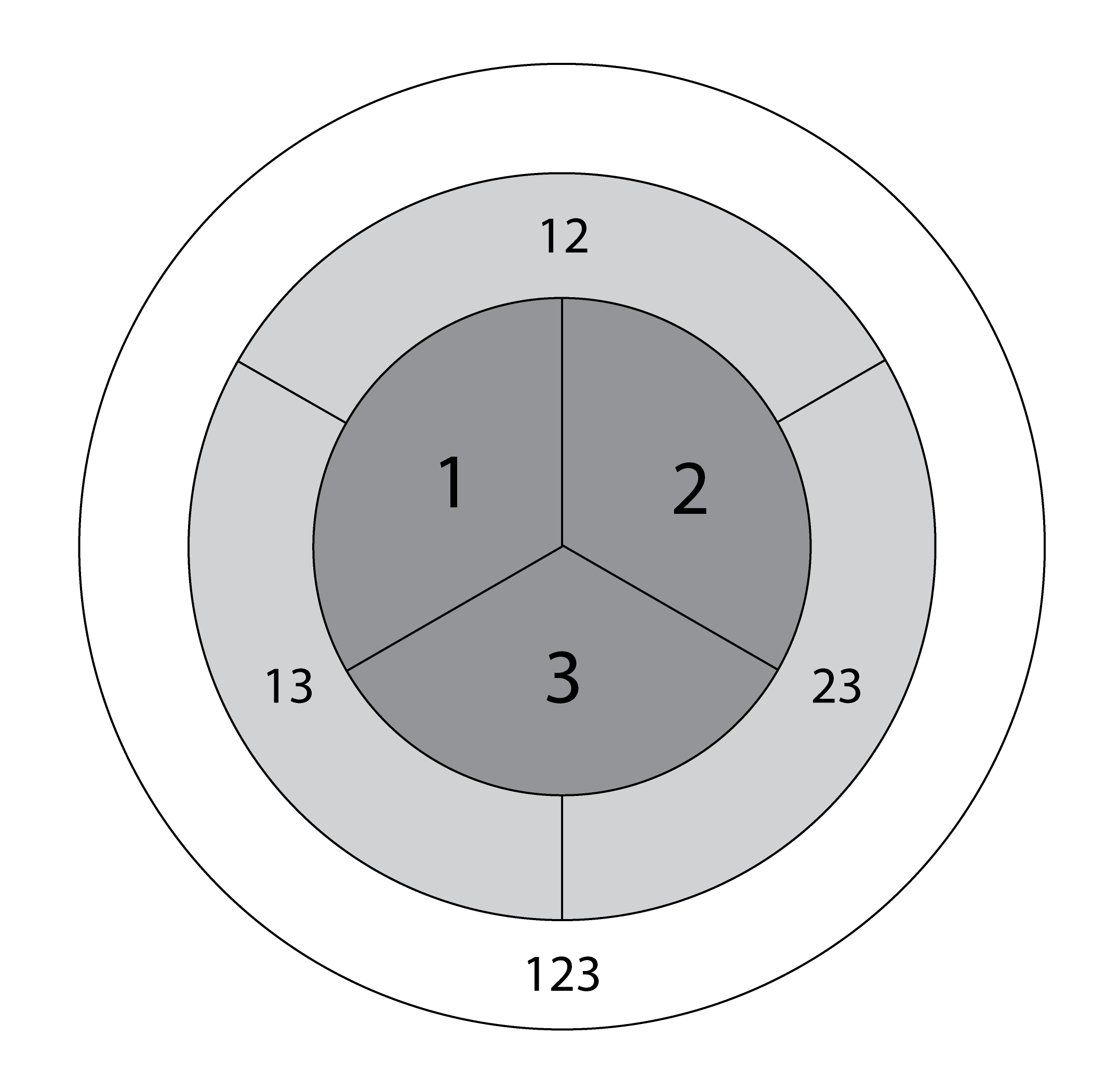}
\end{center}
\caption{Graphical representation of the decomposition of a 3-dimensional Bayes space $\mathcal{B}^2 (\lambda)$ into orthogonal subspaces. See Example~\ref{ex:4.1} for a description.\label{fig:1}}
\end{figure}

\subsection{Consequences of the decomposition of $f_\mu$\label{sec:4.6}}

In this section, consequences of the decomposition of $f_\mu$ developed in Section~\ref{sec:4.3} in combination with Theorem~\ref{thm:4.2} are drawn. This leads to the following multidimensional analog of Proposition~\ref{prop:3.3}. 
\begin{proposition}
\label{prop:4.4}
For any $f_\mu \in \mathcal{B}^2 (\lambda)$, the following statements hold true.
\begin{enumerate}
\item[(i)] 
\emph{Pythagoras' Theorem:} One always has
\[
\| f_\mu\|^2_{\mathcal{B}^2 (\lambda)} = \| f_{\mu,\rm{ind}}\|^2_{\mathcal{B}^2 (\lambda)} + \sum_{I \subseteq D, |I| \ge 2} \| f_{\mu, I,\rm{int}}\|^2_{\mathcal{B}^2 (\lambda)},
\]
where $ \| f_{\mu,\rm{ind}}\|^2_{\mathcal{B}^2 (\lambda)} = \| f_{\mu,1}\|^2_{\mathcal{B}^2 (\lambda)} + \cdots + \| f_{\mu,d}\|^2_{\mathcal{B}^2 (\lambda)}$.
\item[(ii)] 
\emph{Margin-free property:} 
For any given set $I \subseteq D$ with $|I| \ge 2$, let $f_\nu =f_{\mu, I, {\rm int}}$. Then for each $i \in D$, the $i$-th univariate geometric margin $f_{\nu,i}$ of $f_\nu$ satisfies $f_{\nu,i} \prop2 1$. Moreover, for any set $J \subsetneq I$ with $|J| \ge 2$, the $J$-th geometric margin of $f_\nu$ satisfies $f_{\nu, J} \prop2 1$.
\item[(iii)] 
\emph{Independence:}
If $f_\mu = f_{\mu_1} \cdots f_{\mu_d} \in \mathcal{B}^2_{\rm ind}(\lambda)$, then for each $i \in D$, $f_{\mu,i} \prop2 f_{\mu_i}$ and 
\[
f_\mu \in \mathcal{B}^2_{\rm ind}(\lambda) \quad \Longleftrightarrow \quad f_\mu \prop2 f_{\mu,1} \oplus \cdots \oplus f_{\mu,d} \quad \Longleftrightarrow \quad  \forall_{I \subseteq D, |I| \ge 2} \; f_{\mu, I , {\rm int}} \prop2 1.
\] 
\item[(iv)] 
\emph{Yule perturbation:} 
For $g \prop2 g_1 \cdots g_d \in  \mathcal{B}^2_{\rm ind}(\lambda)$, set $h = f_\mu \oplus g$. Then the independence and interaction parts of $h$ satisfy 
\[
h_{\rm ind} = f_{\mu, {\rm ind}} \oplus g, \quad h_{\rm int} = f_{\mu, {\rm int}},
\]
where 
\[
f_{\mu, \rm int} = \bigoplus_{I \subseteq D, | I | \ge 2} f_{\mu, I, {\rm int}}.
\]
In particular, the geometric margins of $h$ satisfy $h_i\prop2 f_{\mu,i} g_i$ for all $i \in D$.
\end{enumerate}
\end{proposition}

Finally, the decomposition of $f_\mu$ in Section~\ref{sec:4.3} can be mirrored into the clr space in an obvious way. From Proposition~\ref{prop:4.2}, any $f_\mu \in \mathcal{B}^2 (\lambda)$ satisfies 
\[
\clr (f_\mu) =  \clr (f_{\mu, {\rm ind}}) + \clr (f_{\mu, \rm int}),
\]
where 
\[
\clr (f_{\mu, {\rm ind}}) = \sum_{ i \in D} \clr (f_{\mu, i}), \quad \clr (f_{\mu, \rm int}) = \sum_{I \subseteq D, | I | \ge 2} \clr (f_{\mu, I, {\rm int}})
\]
and for any $I \subseteq D$ with $| I | \ge 2$, 
\[
\clr (f_{\mu, I, {\rm int}} ) =  \sum_{J \subseteq I, J \neq \varnothing}   (-1)^{|I \setminus J|}\clr (f_{\mu, J}) = \clr (f_{\mu, I})- \sum_{J \subsetneq I, |J| \ge 2} \clr (f_{\mu, J, {\rm int}}) - \sum_{i\in I} \clr (f_{\mu,i}),
\]
where the last equality follows from Proposition~\ref{prop:4.3}. Furthermore, 
\[
\| \clr (f_\mu)  \|^2_{L_2(\lambda)} = \sum_{i=1}^d \| \clr (f_{\mu,i})\|^2_{L^2 (\lambda)} + \sum_{I \subseteq D, |I| \ge 2} \|\clr (f_{\mu, I,\rm{int}})\|^2_{L^2 (\lambda)}
\]
and for any $I \subseteq D$ with $|I| \ge 2$,
\[
\|\clr (f_{\mu, I,\rm{int}})\|^2_{L^2 (\lambda)} = \sum_{i\in I} \| \clr (f_{\mu,i})\|^2_{L^2 (\lambda)} + \sum_{J \subsetneq I, |J| \ge 2} \|\clr (f_{\mu, J,\rm{int}})\|^2_{L^2 (\lambda)}.
\]

\section{Applications to probability densities\label{sec:5}}

This section explores implications of the results in Section~\ref{sec:4} to probability densities. Given a continuous random vector $\boldsymbol{X} = (X_1, \ldots, X_d)^\top$ with cumulative distribution function $F$, let $f$ be its Lebesgue density $f$ and for each $i \in D$, let $F_i$ and $f_i$ respectively denote the cumulative distribution function and Lebesgue density of the $i$th component $X_i$ of $\boldsymbol{X}$. Furthermore, set $\boldsymbol{F} = (F_1, \ldots, F_d)$.

Sklar's decomposition theorem \citep{Sklar:1959, Nelsen:2006} implies the existence of a function $C: [0,1]^d \to [0,1]$ such that, for all $x_1,\ldots, x_d \in \mathbb{R}$,
\[
\Pr(X_1 \le x_1,\ldots, X_d \le x_d) = C\{ F_1(x_1),\ldots, F_d(x_d)\}.
\]
The map $C: [0,1]^d \to [0,1]$, called a copula, is a multivariate distribution function with standard univariate margins. Assuming that $C$ has a Lebesgue density $c$, one then has that, for almost all $x_1, \ldots, x_d \in \mathbb{R}$, 
\begin{equation}
\label{eq:5.1}
f(x_1,\ldots, x_d) = c\{ F_1(x_1),\ldots, F_d(x_d)\} \prod_{i=1}^d f_i(x_i).
\end{equation} 
In Sections~\ref{sec:5.1} and \ref{sec:5.2}, the methodology developed in this paper is investigated for two different choices of reference measure $\lambda$.

\subsection{Densities with respect to the Lebesgue measure\label{sec:5.1}}

Suppose that the reference measure $\lambda$ is chosen to be the Lebesgue measure on $\mathbb{R}^d$ endowed with the usual Borel $\sigma$-field $\mathbb{B}^d$. Then $\lambda = \lambda_1 \times \cdots \times \lambda_d$, as required in Section~\ref{sec:4}, where for each $i \in D$, $\lambda_i$ is the Lebesgue measure on $(\mathbb{R},\mathbb{B})$. As the Lebesgue measure on $(\mathbb{R},\mathbb{B})$ is not finite, the results on Bayes spaces developed in that section are then only applicable when the following assumption holds.

\begin{assumption}
\label{ass:5.1}
The $d$-variate random vector $\boldsymbol{X}$ has a support of the form $\Omega = \Omega_1 \times \cdots \times \Omega_d$ such that $\lambda_i(\Omega_i) < \infty$ for all $i \in D$, and $\boldsymbol{X}$ has a $\lambda$-density $f$ such that $\ln (f)$ is defined almost everywhere and square-integrable on $\Omega$.
\end{assumption}

Note that in practice, Assumption~\ref{ass:5.1} is less restrictive than it may seem. As frequently done in functional data analysis, densities are typically truncated so as to be supported on a hyper-rectangle.  For example, \cite{Hron/etal:2020} proceed this way to investigate the bivariate Gaussian distribution and analyze real data. 

When the random variables $X_1,\ldots, X_d$ are mutually independent, Proposition~\ref{prop:4.4} (iii) implies that the ordinary (or arithmetic) margins $f_1, \ldots, f_d$ of $f$ are proportional to the univariate geometric margins of $f$, denoted $f_{\mu, 1}, \ldots, f_{\mu, d}$, where ${\rm \mu}$ stands for the probability measure associated with $\boldsymbol{X}$.

When the variables $X_1,\ldots, X_d$ are independent, the link between their marginal densities and their geometric margins is more subtle. In addition to Assumption~\ref{ass:5.1}, suppose that for each $i \in D$, $\ln (f_i)$ is square-integrable on $\Omega_i$. In view of Eq.~\eqref{eq:5.1}, the copula density $c$ is then strictly positive almost everywhere on $[0,1]^d$ and, by Minkowski's inequality, the map $\ln \{c (F_1, \ldots, F_d) \}$ is square integrable on~$\Omega$.

From Proposition~\ref{prop:4.1} and Eq.~\eqref{eq:5.1}, one has, for every $i \in D$,
\[
f_{\mu, i}=  \exp \left\{ \frac{1}{\lambda_{D \setminus \{i\}}(\Omega_{D \setminus \{ i \}})}  \int_{\Omega_{D \setminus \{i\}}}\Bigl[ \ln \{ c (F_1,\ldots, F_d) \} + \sum_{j=1}^d \ln (f_j) \Bigr] d \lambda_{D \setminus \{ i \}} \right\}.
\]
Because for any distinct $i, j \in D$, the integral of $\ln(f_j)$ over $\Omega_{D \setminus \{i\}}$ is a constant, one finds that $f_{\mu, i} \prop2 f_i g_i$, where
\[
 g_i =  \exp \left\{ \frac{1}{\lambda_{D \setminus \{i\}}(\Omega_{D \setminus \{i\}})}  \int_{\Omega_{D \setminus \{i\} }} \ln \{ c (F_1,\ldots, F_d) \} d \lambda_{D \setminus \{ i \} } \right\}
\]
depends on both the copula and on the marginals $F_1,\ldots, F_d$. A concrete illustration of this fact is provided below.

\begin{example}
\label{ex:5.1}
Consider the bivariate Gaussian copula with dependence parameter $\rho \in (-1, 1)$,  i.e., a bivariate distribution with density given, for all $u_1, u_2 \in (0,1)$, by
\[
c_\rho (u_1, u_2) = \frac{1}{\sqrt{1- \rho^2}} \exp \left[ -\frac{1}{2(1-\rho^2)} \, \{ \rho^2 (x_1^2 + x_2^2) - 2 \rho x_1 x_2 \} \right] ,
\]
where $x_1 = \Phi^{-1} (u_1)$, $x_2 = \Phi^{-1} (u_2)$, and $\Phi^{-1}$ stands for the quantile function of the univariate Gaussian distribution with zero mean and unit variance, denoted $\mathcal{N} (0, 1)$.

Straightforward calculations based on the formulas given in Section~\ref{sec:3} yield algebraically closed expressions for the geometric margins and the interaction of the Gaussian copula. Namely, one finds, for all $u_1, u_2 \in (0,1)$ and $i \in \{ 1, 2 \}$, 
\[
\clr (c_{\rho, i}) (x_i) =  -\frac{\rho^2}{2 (1 - \rho^2)} \Bigl[  \{ \Phi^{-1}(u_i) \} ^2-1\Bigr]
\]
for the two geometric margins and
\[
\clr (c_{\rho,{\rm int}}) (u_1, u_2) = \frac{\rho}{1-\rho^2} \, \Phi^{-1} (u_1) \Phi^{-1} (u_2)
\]
for the interaction. It is clear from these calculations that the geometric margins of $c_\rho$ depend on the value of $\rho$, and hence differ from the arithmetic margins unless $\rho = 0$. The latter case corresponds to independence; one then has $c_\rho=1$. These calculations can also be obtained as a special case of Example~\ref{ex:5.2} below. 
\end{example}

The calculations in Example~\ref{ex:5.1}, which will also obtain as a special case of Example~\ref{ex:5.2} below, are in line with those derived for the truncated normal case considered by~\cite{Hron/etal:2020}, where additional constants are present due to the need to restrict the density to a finite support.

\subsection{Densities with respect to the product of their margins\label{sec:5.2}}

Choosing the Lebesgue measure as the reference measure in Section~\ref{sec:5.1} restricts the random vector $\boldsymbol{X}$ to have a bounded support. Other than by truncation, this limitation can be overcome by choosing a different reference measure. To this end, consider the following assumption. 

\begin{assumption}
\label{ass:5.2}
The $d$-variate random vector $\boldsymbol{X}$ has a support of the form $\Omega = \Omega_1 \times \cdots \times \Omega_d$ and is absolutely continuous with respect to the Lebesgue measure $\lambda$. Furthermore, for every $i \in D$, the marginal density $f_i$ of $X_i$ is strictly positive on $\Omega_i$ and the unique copula $C$ of $\boldsymbol{X}$ has a $\lambda$-density $c$ such that $\ln (c)$ is defined almost everywhere and square-integrable on $[0,1]^d$.
\end{assumption}

Note that in contrast to Assumption~\ref{ass:5.1}, Assumption~\ref{ass:5.2} can be satisfied by a random vector $\boldsymbol{X}$ without its support being necessarily bounded. Let $\pi_{\boldsymbol{X}}$ denote the probability measure of the random vector $\boldsymbol{X}$ and, for each $i \in D$, let $\pi_{X_i}$ be the probability measure corresponding to its $i$th component $X_i$. Further set
\[
\pi_{\perp} = \pi_{X_1} \otimes \cdots \otimes \pi_{X_d}.
\]
Given that the copula $c$ density is strictly positive almost everywhere on $[0,1]^d$, it then follows from Eq.~\eqref{eq:5.1} that $\pi_{\boldsymbol{X}}$ is equivalent to $\pi_\perp$ and that it has a $\pi_\perp$-density $c \circ \boldsymbol{F}$ given, for all $(x_1,\ldots, x_d) \in \Omega$ by 
\begin{equation}
\label{eq:5.2}
(c \circ \boldsymbol{F}) (x_1, \ldots, x_d) = c \{ F_1 (x_1),\ldots, F_d (x_d) \}.
\end{equation}
Upon changing variables, one can see that
\[
\int_{\Omega} \{\ln(c \circ \boldsymbol{F})\}^2 d \pi_\perp = \int_{[0,1]^d} \{\ln(c)\}^2 d \lambda < \infty
\]
because of Assumption~\ref{ass:5.2}, and hence $\pi_{\boldsymbol{X}}$ --- or equivalently $c \circ \boldsymbol{F}$ --- is an element of the Bayes space $\mathcal{B}^2(\pi_\perp)$ on $\Omega$ with reference measure $\pi_\perp$. 

At the same time, the copula density $c$ is an element of the Bayes space $\mathcal{B}^2(\lambda)$ on $[0,1]^d$ with reference measure $\lambda$. For any set $I \subseteq D$, let $c_I$ denote the $I$-th geometric margin of $c$ as per Definition~\ref{def:4.1}.  From Proposition~\ref{prop:4.2}, $c$ can be decomposed as 
\[
c = c_{\rm ind} \oplus \bigoplus_{I \subseteq D , |I| \ge 2} c_{I, \rm int}.
\]
In this expression, $c_{\rm ind}$ refers to the independence part of $c$ while for any set $I \subseteq D$ with $|I| \ge 2$, $c_{I, \rm int}$ refers to the $I$-th interaction part of $c$, viz.  
\[
c_{\rm ind} = \bigoplus_{i=1}^d c_i, \quad c_{I, \rm int} =\bigoplus_{J \subseteq I, J \neq \varnothing}  \{ (-1)^{|I \setminus J|}\} \odot c_{J}.
\]

Now observe that by a change of variables, one has that for any set $I \subsetneq D$ with $|I| \ge 2$, the $I$-th geometric margin $(c \circ \boldsymbol{F})_I$ of $c \circ \boldsymbol{F}$ is given by
\[
(c \circ \boldsymbol{F}) = \exp \left\{ \int_{\Omega_{D \setminus I}} \ln (c \circ \boldsymbol{F}) d \pi_{D \setminus I} \right\} = c_I\circ \boldsymbol{F},
\]
where for any $(x_1,\ldots, x_d) \in \Omega$, $c_I\circ \boldsymbol{F}(x_1,\ldots, x_d) = c_I \{F_1(x_1),\ldots, F_d(x_d) \}$. Proposition~\ref{prop:4.2} then directly implies the following result.

\begin{proposition}
\label{prop:5.1}
Suppose that an absolutely continuous $d$-variate random vector $\boldsymbol{X}$ with univariate margins $F_1, \ldots, F_d$ and unique copula density $c$ satisfies Assumption~\ref{ass:5.2}. Then the map $c\circ \boldsymbol{F}$ defined by Eq.~\eqref{eq:5.2} can be decomposed as
\[
c\circ \boldsymbol{F} = c_{\rm ind}\circ \boldsymbol{F} \oplus \bigoplus_{I \subseteq D , |I| \ge 2} c_{I, \rm int} \circ \boldsymbol{F}.
\]
\end{proposition} 

What Proposition~\ref{prop:5.1} effectively means is that when working with the reference measure $\pi_\perp$, the decomposition $c \circ \boldsymbol{F}$ of the density of $\boldsymbol{X}$ with respect to $\pi_\perp$ can be computed in two steps, in the same spirit as Sklar's representation theorem.
\begin{itemize}
\item [] \textbf{Step 1:}
Derive the orthogonal decomposition of the density $c$ of the random vector $\boldsymbol{U} = (F_1 (X_1),\ldots, F_d (X_d))$ whose distribution is the copula $C$ of $\boldsymbol{X}$. 
\item [] \textbf{Step 2:}
Deduce the orthogonal decomposition of the density of $\boldsymbol{X}$ by injecting its marginal distributions $F_1, \ldots, F_d$ into the decomposition of $c$. 
\end{itemize}

The advantage of working with copulas is thus two-fold. First, the requirement that $\boldsymbol{X}$ has a bounded support is not needed given that the Lebesgue measure is a probability measure on $[0,1]^d$. Second, the copula $C$ does not depend on the individual behavior of the components $X_1, \ldots, X_d$ of $\boldsymbol{X}$. 

This is illustrated below for the multivariate Gaussian copula. 

\begin{example}
\label{ex:5.2}
As an extension of Example~\ref{ex:5.1} to arbitrary dimension $d \ge 2$, consider a uniform random vector $\boldsymbol{U} = (U_1, \ldots, U_d)^\top$ whose distribution is the Gaussian copula with invertible $d \times d$ correlation matrix~$\Sigma$. If $I_d$ denotes the identity matrix in dimension $d$, the joint density of $\boldsymbol{U}$ is given, for all $u_1, \ldots, u_d \in (0,1)$, by  
\[
c_{\Sigma} (u_1 ,\dots, u_d) = {\rm det} (2 \pi \Sigma)^{-1/2} \exp \left( -\frac{1}{2} \, \boldsymbol{x}^\top A \boldsymbol{x} \right),
\]
where $A = \Sigma^{-1} - I_d$ and $\boldsymbol{x} = (x_1, \ldots, x_d)^\top = (\Phi^{-1} (u_1), \ldots, \Phi^{-1} (u_d))^\top$ with $\Phi^{-1}$ standing for the quantile function of the $\mathcal{N} (0,1)$ distribution, as in Example~\ref{ex:5.1}. 

A straightforward change of variables implies that, up to an additive constant,
\begin{align*}
\int_{[0,1]^d} \ln \{ c_\Sigma (u_1, \ldots, u_d) \} du_1 \cdots du_d & = -\frac{1}{2} \int_{\mathbb{R}^d} \boldsymbol{x}^\top A \, \boldsymbol{x} \left( \prod_{i=1}^d e^{-x_i^2/2} \right) dx_1 \cdots dx_d \\
& = - \frac{1}{2} \, {\rm E} (\boldsymbol{X}^\top A\boldsymbol{X}),
\end{align*}
where the components of the random vector $\boldsymbol{X} = (X_1, \ldots, X_d)^\top$ are mutually independent and $\mathcal{N} (0, 1)$. Accordingly, the above integral reduces to
\[
- \frac{1}{2} \, \mathrm{tr} [ \mathrm{E} \{ \boldsymbol{X}^\top A \, \boldsymbol{X} \}] = - \frac{1}{2} \, \mathrm{tr} (A),
\]
and hence the clr transform of $c_\Sigma$ is given, for all $u_1, \ldots, u_d \in (0, 1)$, by
\[
\clr(c_\Sigma)(u_1,\ldots, u_d) = - \frac{1}{2} \, \boldsymbol{x}^\top A \, \boldsymbol{x} + \frac{1}{2} \, \mathrm{tr} (A).
\]

Letting $A = (a_{ij})$, one can also check easily that for any $i \in D$, one has
\[
\mathrm{E} ( \boldsymbol {X}^\top A \boldsymbol {X} \mid X_i) = a_{ii} \{ \Phi^{-1}(u_i)\}^2 + \mathrm{tr} (A) - a_{ii},
\]
and for arbitrary nonempty set $I \subsetneq D$ and $X_I = \{ X_i: i \in I \}$, one also has
\[
\mathrm{E} (\mathbf{X}^\top A \mathbf{X} \mid X_I) = \sum_{i \notin I} a_{ii} \{\Phi^{-1}(u_i)\}^2 + 2 \sum_{i \notin I}\sum_{j \in I} \Phi^{-1}(u_i)\Phi^{-1}(u_j)
+ \sum_{i \notin I} a_{ii}.
\]
It follows that, for every distinct $i, j \in D$ and for all $u_i, u_j \in [0,1]$, one has
\[
\clr(c_i)(u_i) = -\frac{1}{2} \, [ a_{ii} \{ \Phi^{-1} (u_i) \}^2 - a_{ii} ],
\]
and $\clr(c_{i,j})(u_i, u_j) = -a_{ij}\Phi^{-1} (u_i) \Phi^{-1} (u_j)$, while for arbitrary integer $k \in \{ 3, \ldots, d\}$ and $1 < i_1 \leq \cdots \leq i_k < d$, $\clr(c_{i_1, \ldots, i_k}) \equiv 0$.

This result could perhaps be expected, given that the dependence structure of the random vector $\boldsymbol{U}$ is completely characterized by the pair-wise interactions between its components. See, e.g., \cite{Genest/etal:2007}, for an analogous finding concerning the asymptotic power of rank-based tests of multivariate independence constructed from combinations of Cram\'er--von Mises statistics derived from a M\"obius decomposition of the empirical copula process.
\end{example}

\subsection{Connection with Fr\'echet spaces\label{sec:5.3}}

Working with the reference measure $\pi_\perp$ as was done in Section~\ref{sec:5.2} makes sense for all $d$-variate random vectors $\boldsymbol{X}$ sharing the same univariate marginals $F_1,\ldots, F_d$. The class of all such random vectors is called a Fr\'echet class \citep{Joe:1997}. 

Within the Fr\'echet class with margins $F_1, \ldots, F_d$, let $\mathcal{F}$ denote those distributions that satisfy Assumption~\ref{ass:5.2}. The set of all densities of elements in this class, viz.
\[
\mathcal{B}^2_{\mathcal{F}}(\pi) = \Bigl\{ f \in \mathcal{B}^2(\pi) \; {\Bigl |} \; \int_{\Omega_{D\setminus\{i\}}} f d \pi_{D\setminus \{i\}} \prop2 1 \Bigr\},
\]
is then a subset of the Bayes space $\mathcal{B}^2(\pi)$. Indeed, the requirement 
\[
\int_{\Omega_{D\setminus\{i\}}} f d \pi_{D\setminus \{i\}} \prop2 1
\]
guarantees that elements in $\mathcal{B}^2_{\mathcal{F}}(\pi)$ are in fact proportional to probability distributions in $\mathcal{F}$ whose copula $C$ has a $\lambda$-density $c$ which is positive almost everywhere on $[0,1]$ and is such that $\ln(c)$ is square integrable on $[0,1]^d$. 

Moreover, the set $\mathcal{B}^2_{\mathcal{F}}(\pi)$ is in one-to-one correspondence with the subset
\[
\mathcal{B}^2_{\mathcal{C}} (\lambda) = \Bigl\{ f \in \mathcal{B}^2(\lambda) \; {\Bigl |} \; \int_{[0,1]^{d-1}} f d \lambda_{D\setminus \{i\}} \prop2 1 \Bigr\}
\]
of the Bayes space $\mathcal{B}^2(\lambda)$ on $[0,1]^d$ with reference measure $\lambda$. This correspondence is ensured by the map $\varphi : \mathcal{B}^2_{\mathcal{F}}(\pi) \to \mathcal{B}^2_{\mathcal{C}} (\lambda)$ defined by
\[
f \mapsto \varphi(f) = f(F_1^{-1}, \ldots, F_d^{-1}). 
\]
This map has an inverse $\varphi^{-1}: \mathcal{B}^2_{\mathcal{C}} (\lambda) \to \mathcal{B}^2_{\mathcal{F}}(\pi)$ given by $c \mapsto  \varphi^{-1}(c) = c(F_1,\ldots, F_d)$. This is because for each $i \in D$, the map $F_i$ is continuous and strictly increasing on $\Omega_i$, so that one has $F_i \circ F_i^{-1} (u) = u$ for all $u \in [0,1]$ as well as $F_i^{-1}\circ F_i(x) = x$ for all $x \in \Omega_i$.  It is easily seen that $\varphi$ is in fact an isometry.

\section{Conclusions and outlook\label{sec:6}}

\cite{Hron/etal:2020} have recently opened new perspectives for dependence modeling in a functional analysis context by introducing an orthogonal decomposition of bivariate probability densities into an independent and a dependent part. Their work, which is cast in the theory of Bayes spaces laid out by~\cite{Egozcue/etal:2006} and \cite{Boogaart/etal:2010, Boogaart/etal:2014}, takes its roots in the foundational work of \cite{Aitchison:1982} on the analysis of compositional data. 

These efforts were further expanded here by interpreting the components in the orthogonal decomposition of \cite{Hron/etal:2020} as projections, and by showing how a decomposition formula due to \cite{Kuo/etal:2010} can be used to extend these findings to the multivariate case. Any strictly positive $d$-variate density whose logarithm is square-integrable with respect to a finite reference measure could then be decomposed in a unique way into a direct sum of $2^d - 1$ mutually orthogonal projections. 

The terms in the expansion derived here, which are inspired by the Hoeffding--Sobol and M\"obius decomposition formulas \citep{Mercadier/etal:2022}, are analogous to main effects, called geometric margins by \cite{Hron/etal:2020}, and interactions of all possible orders in analysis of variance. A connection was also made between this approach and the decomposition theorem of \cite{Sklar:1959} which is at the root of the popular copula modeling methodology.

The results reported in this paper provide an analytical framework within which multivariate probability densities viewed as data objects could be modeled. As already transpired from the work of \cite{Hron/etal:2020}, this approach looks quite promising for the acquisition of new theoretical insights and the design of innovative inferential methods for functional data analysis of dependence structures. The few examples discussed herein provide a proof of concept in the multivariate case. 

Access to efficient computational techniques adapted to the multivariate Bayes space framework will be a key element to the future successful deployment of this methodology. The B-spline representation used by \cite{Hron/etal:2020} to represent the clr transformation of bivariate densities seems well suited to the task and efforts are currently underway to extend it to the multivariate context.  

\section*{Acknowledgments}

Funding in support of this work was provided by the Canada Research Chairs Program (Grant 950--231937), the Natural Sciences and Engineering Research Council of Canada (Grants RGPIN--2016--04720 and RGPIN-2022-03614), and the Czech Science Foundation (Grantov\'a Agentura \v{C}esk\'e Republiky, Grant  22-15684L). Part of this work was completed while the second author visited the Centre de recherches math\'ematiques (CRM) in Montr\'eal (Qu\'ebec, Canada) in June 2022 as part of the CRM--Simons Scholar-in-Residence program.


\bibliographystyle{apalike}
\bibliography{ref}

\end{document}